\documentclass{article}
\usepackage{amsmath}

\usepackage{amsmath}
\usepackage{amssymb}
\usepackage{amscd}
\usepackage{amsthm}
\usepackage[matrix,frame,import,curve]{xypic}

\def\ZZ{{\mathbb Z}}

\def\QQ{{\mathbb Q}}
\def\PP{{\mathbb P}}
\def\LL{{\mathbb L}}
\def\CC{{\mathbb C}}
\def\cX{{\mathcal X}}

\def\cA{{\mathcal A}}
\def\cO{{\mathcal O}}
\def\cT{{\mathcal T}}
\def\cQ{{\mathcal Q}}
\def\cI{{\mathcal I}}
\def\bk{{\bf k}}
\def\bd{{\bf d}}

\def\Spec{{\bf {Spec}}}
\def\RHom{{\mathbf{R}\rm{Hom}}}

\def\D{\mathrm{D}}
\def\Ext{{\mathrm{Ext}}}
\def\Hom{{\mathrm{Hom}}}
\def\End{{\mathrm{End}}}

\def\O{\mathcal{O}}
\def\P{\mathcal{P}}
\def\E{\mathcal{E}}
\def\F{\mathcal{F}}
\def\L{\mathcal{L}}

\def\sH{\underline{H}}

\newtheorem{lemma}{Lemma}[section]
\newtheorem{theorem}[lemma]{Theorem}
\newtheorem{claim}[lemma]{Claim}
\newtheorem{corollary}[lemma]{Corollary}

\newtheorem{prop}[lemma]{Proposition}
\theoremstyle{definition}
\newtheorem{definition}[lemma]{Definition}
\newtheorem{example}[lemma]{Example}

\newtheorem{remark}[lemma]{Remark}

\theoremstyle{remark}
\newtheorem*{proof*}{Proof}
\numberwithin{equation}{section}

\begin{document}
\title{Chern-Simons functions on toric Calabi-Yau threefolds and Donaldson-Thomas theory}
\author{Zheng Hua\footnote{Address: Department of Mathematics,
University of British Columbia, Vancouver, V6T 1Z2, Canada, E-mail:
\tt hua@math.ubc.ca}{}}

\maketitle
\begin{abstract}
In this paper, we give a construction of the global Chern-Simons
functions for toric Calabi-Yau stacks of dimension three using
strong exceptional collections. The moduli spaces of sheaves on such
stacks can be identified with critical loci of these functions. We
give two applications of these functions. First, we prove Joyce's
integrality conjecture of generalized DT invariants on local
surfaces. Second, we prove a dimension reduction formula for virtual
motives, which leads to two recursion formulas for motivic
Donaldson-Thomas invariants.
\end{abstract}

\section{Introduction}
Moduli spaces of sheaves (more general, complexes of sheaves) on
Calabi-Yau threefolds are examples of moduli problems with symmetric
obstruction theories (\cite{Behrend}). It is expected that such a
moduli space is locally critical set of a holomorphic function. Such
functions are called Chern-Simons (CS) functions. Chern-Simons
functions play an important role in Calabi-Yau geometry because
Behrend proved the Milnor number of CS function is the microlocal
version of Donaldson-Thomas invariant (See \cite{Behrend}).

In a seminal work \cite{JS}, Joyce and Song proved the existence of
CS functions for moduli spaces of stable sheaves on compact
Calabi-Yau 3-folds using analytic techniques in gauge theory. In
this paper, we give a different construction of the CS functions on
toric Calabi-Yau 3-folds. Our construction has a few new
ingredients. First, the functions we construct are algebraic.
Second, the moduli spaces of stable sheaves are in fact globally
critical sets of these functions. Third, the construction is
explicit, i.e. there is an algorithm to write down such functions
starting with a toric CY 3-fold together with some extra data (See
end of Section \ref{sec_CS}). Another byproduct is that we are able
to make the following heuristic statement mathematically rigorous:
\begin{claim}
Let $X$ be a toric Fano stack of dimension two and $Y$ be the total
space of canonical bundle over $X$. Let $\gamma$ be a class in the
Grothendieck group of $X$, which is a subgroup of the Grothendieck
group of $Y$. Denote the moduli spaces of sheaves on $X$ (resp. $Y$)
with class $\gamma$ by $\mathfrak{M}_\gamma(X)$ (resp.
$\mathfrak{M}_\gamma(Y)$). With suitable derived schemes (stacks)
structures we have the isomorphism
\[
\mathfrak{M}_\gamma(Y)\simeq T^*[-1]\mathfrak{M}_\gamma(X)
\] where the righthand side is the odd cotangent bundle over
$\mathfrak{M}_\gamma(X)$.
\end{claim}

The construction of CS function consists of three steps:
\begin{enumerate}
\item Find a good $t$-structure in the derived category $\D^b(Y)$.
The heart of this $t$-structure is the abelian category of
representations of a quiver with relations. Such an abelian category
is good in the sense that it has finite projective dimension.
\item On a moduli space of representations with fixed dimension vector,
we find a \emph{maximally degenerate} point, which corresponds to
the semi-simple representations. The tangent complex of the moduli
space at this point is given by the well studied $L_\infty$
($A_\infty$) Yoneda algebra in representation theory. We compute the
$L_\infty$( $A_\infty$) products and prove they are bounded. The
Calabi-Yau condition defines a cyclic pairing on this $L_\infty$
algebra, which together with the $L_\infty$ products determine the
CS function.
\item Embed the moduli spaces of sheaves into the moduli spaces of
representations as open substacks.
\end{enumerate}

Step one is based on the existence of full strong exceptional
collections of line bundles on toric Fano stacks of dimension two
(Theorem \ref{King's}). This is proved in a work of the author and
Boriosv \cite{BorHua}. Passing from strong exceptional collection to
associated quiver is a consequence of derived Morita equivalence. We
will study this in Section \ref{sec_morita}.

Step two is based on the cyclic completion Theorem \ref{Yoneda
cyclic} and boundness of $L_\infty$ products (Theorem \ref{boundness
of mu_k}). Theorem \ref{Yoneda cyclic} is first proved by Aspinwall
and Fidkowski in \cite{Aspinwall} and later reproved by Segal in
\cite{Segal}. The terminology \emph{cyclic completion} is due to
Segal. The proofs of these two theorems are in Section
\ref{sec_cylic}.

In Section \ref{sec_CS}, we construct the CS functions and show that
the moduli spaces of sheaves are open substacks of the critical sets
modulo gauge groups. Several examples of CS functions are discussed
in Section \ref{sec_examples}.

The language of $L_\infty$ algebras and derived schemes (stacks)
(developed in \cite{KS}) is extensively used in the paper. Moduli
spaces mentioned above are zero loci of an odd vector field on a
differential graded symplectic manifolds and the CS functions we
construct are essentially Hamiltonian functions associated to it. In
Section \ref{sec_Linfty}, we give a short introduction to $L_\infty$
algebra and dg schemes.

In the last three sections, we give two applications of the CS
function. In Theorem \ref{smoothness} of Section
\ref{sec_Integrality}, we prove the $L_\infty$ products vanish at
semi-stable points of moduli space of sheaves on local surfaces,
which leads to a proof of a special case of the integrality
conjecture of Joyce and Song (\cite{JS}). In Theorem \ref{dimension
reduction}, we prove a dimension reduction formula of virtual
motives for CS functions, which generalizes some results in
\cite{BBS}. By manipulating this dimension reduction formula, we
compute the generating series of two types of motivic
Donaldson-Thomas invariants on toric CY stacks. This is done in
Section \ref{sec_generating series}.

\paragraph{Acknowledgment} I would like to thank Jim Bryan, Kai
Behrend and Kentaro Nagao for interests in my work, comments,
conversations and helpful correspondence. Some results in the last
two sections are also obtained by Andrew Morrison in an independent
work \cite{Morrison}. Thanks to Andrew Morrison for kindly sharing
his preprint. Similar results have also been obtained by Nagao
\cite{Nagao} and Mozgovoy \cite{Moz}. I would like to thank the
referee for pointing out this. Section \ref{sec_Integrality} is
finished during my visit to Imperial college London. It is motivated
by the discussion with Martijn Kool. The author wants to thank
Martijn Kool and Richard Thomas for the very helpful comments. I
also want to thank Dominic Joyce for pointing out a mistake in the
earlier draft.

\paragraph{Notations}
Three dimensional smooth toric Calabi-Yau stacks are in one to one
correspondence with the set of $3$-dimensional cones over convex
lattice polygons $\Delta$ contained in an affine hyperplane,
together with a triangulation of $\Delta$. When the polygon $\Delta$
has at least one interior lattice point, we can consider the
barycentric triangulation. This gives a fan $\Sigma$ on the affine
hyperplane such that its supporting polygon is $\Delta$. The fan
$\Sigma$ determines a $2$-dimensional toric Fano stack $X_\Sigma$
($X$ for short). The cone over $\Sigma$ determines a $3$-dimensional
toric CY stack $Y_\Sigma$ ($Y$ for short), which is the total space
of canonical bundle over $X_\Sigma$. We call such a toric CY 3-stack
a \emph{local surface}. The CY 3-stacks associated to other
triangulations of $\Delta$ are relating to $Y_\Sigma$ by sequences
of flops.

\begin{enumerate}
\item[$\bullet$] $\pi: Y\to X$ is the projection and $\iota:
X\to
Y$ is the inclusion of zero section
\item[$\bullet$] $\D^b(X)$: bounded derived category of coherent
sheaves on $X$
\item[$\bullet$] $\D^b(Y)$: bounded derived category of coherent
sheaves on $Y$
\item[$\bullet$] $\D_\omega$: the full subcategory of $\D^b(Y)$ of
objects with cohomology sheaves supported at $X$
\end{enumerate}

\section{$L_\infty$ algebra and Differential graded
schemes}\label{sec_Linfty}

This is a short introduction to $L_\infty$ algebras and differential
graded schemes. A standard reference for this topic is \cite{KS}.
The readers who are familiar with $\infty$-algebras can skip this
section.

\subsection{$L_\infty$ algebras}

Let $\bk$ be a field.
\begin{definition}
A $L_\infty$ algebra is a graded $\bk$-vector space $L$ with a
sequence $\mu_1,\ldots, \mu_k, \ldots$ of graded antisymmetric
operations of degree 2, or equivalently, homogeneous multi-linear
maps
\[
\mu_k: \wedge^k L \to L[2-k]
\]
such that for each $n > 0$, the $n$-Jacobi rule holds:
\begin{equation*}
    \sum_{k=1}^n (-1)^k \sum_{\substack{ i_1<\ldots<i_k ;
    j_1<\ldots<j_{n-k} \\
    \{i_1,\dotsc,i_k\}\cup\{j_1,\dotsc,j_{n-k}\}=\{1,\dotsc,n\} }}
    (-1)^\epsilon \, \mu_n(\mu_k(x_{i_1},\dotsc,x_{i_k}),x_{j_1},\dotsc,x_{j_{n-k}}) = 0 .
  \end{equation*}
Here, the sign $(-1)^\epsilon$ equals the product of the sign
$(-1)^\pi$
  associated to the permutation
  \begin{equation*}
    \pi = \bigl( \begin{smallmatrix} 1 & \dotso & k & k+1 & \dotso & n
      \\ i_1 & \dotso & i_k & j_1 & \dotso & j_{n-k}
    \end{smallmatrix} \bigr)
  \end{equation*}
  with the sign associated by the Koszul sign convention to the action
  of $\pi$ on the elements $(x_1,\dotsc,x_n)$ of $L$.
\end{definition}

\begin{definition}\label{Maurer-Cartan}
Let $(L,\mu_k)$ be a $L_\infty$ algebra. An element $x\in L^1$ is
called a Maurer-Cartan element if $x$ satisfies the formal equation
(called Maurer-Cartan equation)
\[
\sum_{k=1}^\infty \frac{1}{k!}\mu_k(x,\ldots,x)=0.
\]
If the above formal sum is convergent then there is a map $Q: L^1\to
L^2$, defined by $x\mapsto \sum_{k=1}^\infty
\frac{1}{k!}\mu_k(x,\ldots,x)$. We call it the curvature map. The
set of elements in $L^1$ satisfying the Maurer-Cartan equation is
denoted by $MC(L)$.
\end{definition}
\begin{definition}\label{twisted diff}
Let $L$ be a $L_\infty$ algebra. We write $\delta$ for the first
$L_\infty$ product $\mu_1: L\to L[1]$. It follows from $L_\infty$
relations that $\delta^2=0$. Let $x$ be a Maurer-Cartan element of
$L$. We define the \emph{twisted} differential $\delta^x$ by the
formula $\delta^x(y)=\delta(y)+\sum_{k=2}^\infty
\frac{1}{(k-1)!}\mu_k(x,\ldots,x,y)$. By manipulating the
Maurer-Cartan equation and the $L^\infty$ relations, one can check
that $(\delta^x)^2=0$.
\end{definition}

Given an homogenous element $a\in L$, we denote its grading by
$|a|$.
\begin{definition}
A finite dimensional $L_\infty$ algebra $(L,\mu_k)$ is called
\emph{cyclic} if there exists a homogenous bilinear map
\[
\xymatrix{ \kappa: L\otimes L \ar[r] & \bk[-3]}
\]
satisfies:
\begin{enumerate}
\item[(1)] $\kappa(a,b)=(-1)^{|a||b|}\kappa(b,a)$;
\item[(2)]
$\kappa(\mu_k(a_1,\ldots,a_k),a_{k+1})=(-1)^{|a_1|(|a_2|+\ldots+|a_{k+1}|)}\kappa(\mu_k(a_2,\ldots,a_{k+1}),a_1)$;
\item[(3)]
$\kappa$ is non-degenerated on $H^\bullet(L,\delta)$.
\end{enumerate}
\end{definition}
We call such a $\kappa$ a cyclic pairing on $L$.

\begin{definition}\label{hamiltonian}
Let $(L, \mu_k, \kappa)$ be a cyclic $L_\infty$ algebra. The
Chern-Simons function associated to $L$ is the formal function
\[
f(z) = \sum_{k=1}^\infty \frac{(-1)^{\frac{k(k+1)}{2}}}{(k+1)!}
\kappa( \mu_k(z,\ldots,z), z).
\]
\end{definition}

\subsection{Differential graded schemes}

\begin{definition}
A differential graded (dg for short) scheme $X$ is a pair
$(X^0,\cO^\bullet_X)$, where $X^0$ is an ordinary scheme and
$\cO^\bullet_X$ is a sheaf of $\ZZ^-$ graded commutative dg algebra
on $X^0$ such that:
\begin{enumerate}
\item[(1)] $\cO^0_X=\cO_{X^0}$;
\item[(2)] $\cO^i_X$ are quasi-coherent $\cO_{X^0}$ modules.
\end{enumerate}
\end{definition}
The cohomology sheaves of $\cO^\bullet_X$, denoted by
$\sH^i(\cO^\bullet_X)$ are $\cO_{X^0}$ modules. In particular,
$\sH^0(\cO^\bullet_X)$ is a sheaf of quotient ring of
$\cO_X^0=\cO_{X^0}$. We define the ``0-truncation" of $X$ to be the
ordinary scheme
\[
\pi_0(X)=\Spec\ \sH^0(\O^\bullet_X).
\]
It is a subscheme of $X^0$.
\begin{definition}
A morphism $f$ between dg schemes $X$ and $Y$ is a morphism of
ordinary schemes $f_0: X^0\to Y^0$ together with a morphism of dg
algebras $f_0^* \O^\bullet_Y\to\O^\bullet_X$. A morphism $f$ is
called a quasi-isomorphism if $f$ induces an isomorphism of
$\pi_0(X)$ and $\pi_0(Y)$ and isomorphisms between
$\sH^i(\O^\bullet_X)$ and $\sH^i(\O^\bullet_Y)$.
\end{definition}

\begin{definition}
A dg-scheme $X$ is called smooth (or a dg-manifold) if the following
conditions hold:
\begin{enumerate}
\item[(a)] $X^0$ is a smooth algebraic variety.
\item[(b)] Locally on Zariski topology of $X^0$, we have an isomorphism of
graded algebras $\cO^\bullet_X \simeq Sym _{\cO_{X^0}}  Q^{-1}\oplus
Q^{-2}\oplus \ldots$, where $Q^{-i}$ are vector bundles (of finite
rank) on $X^0$.
\end{enumerate}
\end{definition}

Every $L_\infty$ algebra defines a dg-manifold.
\begin{example}\label{Linfty-dg mfd}
Let $L=L^{-k}\oplus\ldots\oplus L^0\oplus L^1 \oplus\ldots$ be a
finite dimensional $L_\infty$ algebra and $\tau^{>0} L$ be the
truncation of $L$ in positive degrees. Let $X^0$ be the linear
manifold $L^1$ and $\cO^\bullet_X$ be the completed symmetric
algebra $\widehat{Sym\ \tau^{>0}L[1]^*}$, considered as a sheaf over
$L^1$. It has a structure of differential graded algebra (dga). The
$L_\infty$ structure are multi-linear maps $\mu_k: Sym^k L[1]\to
L[2]$. The dual map of $\sum \frac{1}{k!}\mu_k$ defines a derivation
from $\cO^\bullet_X$ to $\cO^{\bullet}_X$ of degree one. We call it
by $q$. The $L_\infty$ relations is equivalent with the condition
that $q^2=0$. It can be interpreted as an odd vector field on the
dg-manifold. The ``0-truncation" $\pi_0(X)$ can be identified with
the Maurer-Cartan locus $MC(L)$. We call the dg-manifold constructed
above the formal dg-manifold associated to $L$.
\end{example}

Given a cyclic $L_\infty$ algebra $(L, \mu_k, \kappa)$, the formal
dg manifold constructed in \ref{Linfty-dg mfd} is a formal
symplectic dg-manifold in the sense of \cite{KS}. The pairing
$\kappa$ can be viewed as an odd sympletic form.

On a formal dg manifold, we can define the analogue of usual Cartan
calculus \cite{KS}. The CS function $f$ is the Hamiltonian function
of the odd vector field $q$ on $X$ with respect to the odd
symplectic form $\kappa$. In particular, $crit(f)$ coincides with
Maurer-Cartan locus of $L$.

\paragraph{Comments on $A_\infty$ and $L_\infty$ algebra}
Given an $A_\infty$ algebra $(R, m_k)$, we can construct, in a
canonical way, a $L_\infty$ algebra $(L, \mu_k)$. This is done by
anti-commuting $m_k$. A lazy way to do that is to first construct a
dg algebra quasi-isomorphic the $R$. Anti-commutize it into a dg Lie
algebra and then take the cohomology. The Maurer-Cartan sets of
$R_\omega$ and $L_\omega$ are the same as sets. In the process of
anti-symmetrization, a cyclic $A_\infty$ algebra goes to a cyclic
$L_\infty$ algebra. We will skip the formal definition of $A_\infty$
algebra (it can be found in \cite{KS}) although it is implicitly
used in the later sections. Using $L_\infty$ algebras has the
advantage that one can make sense of Maurer-Cartan set as a scheme
instead of a non-commutative scheme.

\section{Derived categories of toric
stacks and Morita equivalence}\label{sec_morita}
\begin{definition}\label{exceptional}
Let $\bk$ be a field. Given a $\bk$-linear triangulated category
$\cT$, an object $E\in\cT$ is called \emph{exceptional}, if
$\Ext^i(E,E)=0$ for all $i\neq 0$ and $\Ext^0(E,E)=\bk$.
\begin{enumerate}
\item[$\bullet$]
A sequence of exceptional objects $E_1,\ldots,E_n$ is called an
\emph{exceptional collection} if $\Ext^i(E_j,E_k)=0$ for arbitrary
$i$ when $j>k$.
\item[$\bullet$] An exceptional collection is called
\emph{strong} if $\Ext^i(E_j,E_k)=0$ for any $j$ and $k$ unless
$i=0$.
\item[$\bullet$] We say an exceptional collection is \emph{full}
if it generates $\cT$.
\end{enumerate}
\end{definition}

Let $E,F$ be an exceptional pair in $\cT$. We define the left (resp.
right) mutation $L_E F$ (resp. $R_F E$) using the distinguished
triangles.
\[
\xymatrix{L_E F\ar[r] &\RHom(E,F)\otimes E\ar[r] &F}\]\[
\xymatrix{E\ar[r] &\RHom(E,F)^*\otimes F\ar[r] &R_F E}
\]
Mutations of exceptional collection are exceptional (\cite{Bondal}).
But mutations of strong exceptional collections are not necessary
strong.

Given an exceptional collection $E_0,\ldots,E_n$, we can define
another exceptional collection $F_{-n},F_{-n+1},\ldots,F_0$, called
the \emph{dual} exceptional collection to $E_0,\ldots,E_n$. First
let $F_0$ equal to $E_0$. Second, make $F_{-1}=L_{E_0} E_1$. Then
define $F_{-i}$ inductively by $L_{F_{-i+1}}L_{F_{-i+2}}\ldots
L_{F_0} E_i$.

In our application, $\cT$ will be the bounded derived category
$\D^b(X)$ of a smooth algebraic variety (stack) $X$. The exceptional
objects are always assumed to belong to the heart of certain
$t$-structure.

Given a full strong exceptional colleciton $E_0,\ldots,E_n$, we
denote the direct sum $\bigoplus_{i=0}^n E_i$ by $T$. It is called a
tilting object.
\begin{theorem}\cite{Bondal}\label{morita}
The exact functor $\RHom(T, -)$ induces an equivalence between
triangulated categories $\D^b(X)$ and $\D^b(mod-A)$, where
$A=\End(T)$. This equivalence is usually referred as derived Morita
equivalence.
\end{theorem}
Let $\E$ be an object in $\D^b(X)$, the right $A$-module structure
on $\RHom(T,\E)$ is given by pre-composition. The quasi-inverse
functor of $\RHom(T, -)$ is $-\otimes^\mathbf{L}_A T$.

We can define a quiver with relations from a strong exceptional
collection according the following recipe. First, define the set of
nodes of $\cQ$, denoted by $\cQ_0$ to be the ordered set
$\{0,1,\ldots,n\}$. The $i$-th node corresponds to the generator of
$\Hom(E_i,E_i)$. The set of arrows of $\cQ$, denoted by $\cQ_1$ is
double graded by source and target. The graded piece $\cQ_1^{i,j}$
is a set with cardinality $dim_\CC\Hom(E_i,E_j)$. With a choice of
basis on $\Hom(E_i,E_i)$, the elements in $\cQ_1^{i,j}$ is in one to
one correspondence with such a basis. The exceptional condition
guarantees that there is no arrow that decreases the indices of
nodes. The relations of $\cQ$ are determined by the commutativity of
composition of morphisms. The nodes and arrows generate the free
path algebra $\CC\cQ$, which is spanned as a vector space by all the
possible paths. Multiplication of $\CC\cQ$ is defined by
concatenation of paths. Relations of $\cQ$ form a two-side ideal
$\cI$ of $\CC\cQ$. We call $\CC\cQ/\cI$ the path algebra of
$(\cQ,\cI)$. In some situations, we omit $\cI$ and write just $\cQ$.
It follows from the construction that $\CC\cQ/\cI\simeq A$.

A representation of $(\cQ,\cI)$ is given by the following pieces of
data:
\begin{enumerate}
\item[$\bullet$] Associate a finite dimensional vector spaces $V_i$
to each node $i$;
\item[$\bullet$] Associate a matrix $a^{i,j}$ to each arrow from nodes $i$ to
$j$ such that the matrices associated to any elements in $\cI$ are
zero.
\end{enumerate}
Denote the category of finite dimensional representations of
$(\cQ,\cI)$ by $Rep_\bk (\cQ,\cI)$. There is an equivalence of
abelian categories
\[
Rep_\bk (\cQ,\cI)\cong\CC\cQ/\cI -mod\cong A-mod.
\]

The abelian category $mod- A$ is Noetherian and Artinian. Its simple
objects are exactly those representations that have one dimensional
vector space over node $i$ and $0$ over other nodes. We denote them
by $S_i$. Under the functor $\RHom(T, -)$, the exceptional objects
$E_i$ are mapped to right projective $A$ modules and $F_{-i}$ are
mapped to shifts of simple modules $S_i[-i]$.

The \emph{Yoneda algebra} $R$ of $A$ is defined to be
$\Ext_A^\bullet(\bigoplus_{i=0}^n S_i,\bigoplus_{i=0}^n S_i)$. It
has a canonical $A_\infty$ algebra structure.

Theorem \ref{morita} builds up a link between geometry and
representation theory of quiver suppose one can find a full strong
exceptional collection in $\D^b(X)$. In general, there is no reason
why such collection (even single exceptional object) should exist.
However, the existence result can be proved for toric Fano stack of
dimension two.

Recall that a two dimensional convex lattice polygon $\Delta$ with a
distinguished interior lattice point determines a fan $\Sigma$ by
taking the barycentric triangulation. This determine uniquely a
toric stack, which is denoted by $X_\Sigma$. The Fano condition is
equivalent to the convexity of $\Delta$. We refer to section 3 of
\cite{BorHua} for an introduction to toric stacks.

\begin{theorem}\cite{BorHua}\label{King's}
Let $X_\Sigma$ be a complete toric Fano stack of dimension two. The
bounded derived category of coherent sheaves $\D^b(X_\Sigma)$ has a
full strong exceptional collection consisting of line bundles. The
length of the strong exceptional collection is always equal to the
integral volume of $\Delta$, which is also equal to the Euler
characteristic $\chi(X_\Sigma)$.
\end{theorem}

We will try to extend the derived Morita equivalence to the study of
CY stack $Y$. Consider the exact functor $\RHom(\pi^*T, -)$ from
$\D^b(Y)$ to $\D^b(mod-B)$, where $B=\Hom^\bullet(\pi^*T,\pi^*T)$.
It turns out that this is still an equivalence of triangulated
category if we define the righthand side appropriately. The algebra
$B$ (called the roll-up helix algebra by Bridgeland) in general
carries nontrivial dg-algebra structure. However, in order to apply
the quiver techniques we need to find a strong exceptional
collection such that the differential of $B$ vanishes. This is an
additional condition on strong exceptional collection.

The following proposition is a generalization of Proposition 4.1 of
\cite{Bridgeland}, which is originally proved for $\PP^2$.
\begin{prop}\label{helix cond}
Let $\L_0,\ldots,\L_{n}$ be a full strong exceptional collection of
line bundles on a toric Fano stack of dimension two. The roll-up
(dg)-helix algebra $B$ is in fact an algebra, i.e.
$\Ext^{>0}(\pi^*T,\pi^*T)=0$. Therefore, the exact functor
$\RHom(\pi^*T, -)$ induces an equivalence from $\D^b(Y)$ to
$\D^b(mod-B)$.
\end{prop}
\begin{proof}
We need a technical lemma from \cite{BorHua} about cohomology of
line bundles on toric stacks.

For every ${\bf r}=(r_i)_{i=1}^n\in \ZZ^n$ we denote by ${\rm Supp}
({\bf r})$ the simplicial complex on $n$ vertices $\{1,\ldots,n\}$
which consists of all subsets $J\subseteq \{1,\ldots,n\}$ such that
$r_i\geq 0$ for all $i\in J$ and there exists a cone of $\Sigma$
that contains all $v_i, i\in J$. For example, if all coordinates
$r_i$ are negative then the simplicial complex ${\rm Supp}({\bf r})$
consists of the empty set only, and its geometric realization is the
zero cone of $\Sigma$. In the other extreme case, if all $r_i$ are
nonnegative then the simplicial complex ${\rm Supp}({\bf r})$
encodes the fan $\Sigma$, which is its geometric realization.

\begin{lemma}\label{hp}(Proposition 4.1 \cite{BorHua})
The cohomology $H^p(X_{\Sigma},\L)$ is isomorphic to the direct sum
over all ${\bf r}=(r_i)_{i=1}^n$ such that $\bigoplus(\sum_{i=1}^n
r_i E_i)\cong \L$ with $E_i$ being toric invariant divisors, of the
$({\rm rk}(N)-p)$-th reduced homology of the simplicial complex
${\rm Supp}({\bf r})$.
\end{lemma}

By adjunction,
$\Hom^d(\pi^*T,\pi^*T)=\bigoplus_{k\geq0}\Hom^d_X(T,T\otimes
\omega_X^{-k})$. In order to prove the proposition, it suffices to
show that $H^d(X,\L_i^{-1}\otimes \L_j\otimes \omega_X^{-1})=0$ for
$d=1,2$. Because $\L_0,\ldots,\L_j$ is strong excetpional, we have
$H^d(X,\L_i^{-1}\otimes \L_j)=0$ for $d=1,2$. Consider all the
possible integral linear combinations $\sum_{i=1}^m r_i E_i$ such
that $\cO(\sum_{i=1}^m r_i E_i)=\L_i^{-1}\otimes \L_j$. By Lemma
\ref{hp}, $H^d(X,\L_i^{-1}\otimes \L_j)=0$ for $d=1,2$ means
$\rm{Supp}(\bf{r})$ is contractible. Notice that if
$\rm{Supp}(\bf{r})$ is contractible then $\rm{Supp}(\bf{r}+1)$ is
also contractible. Then by Lemma \ref{hp} again,
$H^d(X,\L_i^{-1}\otimes \L_j\otimes \omega_X^{-1})=0$ for $d=1,2$.
\end{proof}

Now we can write $B$ simply by $\End(\pi^*T)$. It is also the path
algebra of a quiver with relations. This quiver can be constructed
by the same recipe as in the previous section. Let's denote it by
$\cQ_\omega$. Notice that $\cQ_\omega$ will have cyclic paths
because the pull back of exceptional objects will have homomorphisms
in both directions. Again, we have an equivalence of abelian
categories
\[
Rep_\bk (\cQ_\omega,\cI)\cong B-mod.
\]
The path algebra $B$ is naturally graded by the length of the path.
A $B$ module $M$ is called \emph{nilpotent} if there exists $k\gg 0$
such that $B_k M=0$. The exact functor $\RHom(\pi^*T, -)$ maps
$\D_\omega$ to the derived category of nilpotent $B$ modules
$D^b(mod_0-B)$.

The pushforward $\iota_*$ defines an exact functor from $\D^b(X)$ to
$\D_\omega$. Under Morita equivalence, $\iota_*(F_{-i}[i])$ are the
simple modules in $\D^b(mod_0-B)$ corresponding to those one
dimensional representations associated to each vertices of
$\cQ_\omega$.

Similarly, we call the self extension algebra
$$\Ext^\bullet_{B}(\bigoplus_{i=0}^n
\iota_*S_i,\bigoplus_{i=0}^n \iota_*S_i)$$ the Yoneda algebra,
denoted by $R_\omega$. It carries a natural $A_\infty$ structure as
well.

The following reconstruction theorem is due to Keller.

\begin{theorem}\label{ModAtoAinfty}
Let $A$ and $B$ as before. There are fully faithful functors
\[
\D^b(mod-A)\hookrightarrow\D_\infty(R)
\] and
\[
\D^b(mod-B)\hookrightarrow\D_\infty(R_\omega)
\] where $\D_\infty(R)$ and $\D_\infty(R_\omega)$ are derived
categories of modules over $A_\infty$-algebras.
\end{theorem}
\begin{proof}
Theorem 4.3 of \cite{Keller}.
\end{proof}

We give the example of derived morita equivalence on $\PP^2$ and
local $\PP^2$.
\begin{example}\label{P2}
Let $X$ be $\PP^2$. The line bundles $\O,\O(1),\O(2)$ form a full
strong exceptional collection. Take the tilting bundle $T=\O\oplus
\O(1)\oplus \O(2)$. The quiver $\cQ$ is
\begin{equation}
\begin{xy} <1.5mm,0mm>:
  (0,0)*{\circ}="a",(20,0)*{\circ}="b",(40,0)*{\circ}="c",
  (0,-4)*{0},(20,-4)*{1},(40,-4)*{2}
  \ar@{->}@/^3mm/|{x} "a";"b"
  \ar@{->}@/^0mm/|{y} "a";"b"
  \ar@{->}@/_3mm/|{z} "a";"b"
  \ar@{->}@/^3mm/|{x'} "b";"c"
  \ar@{->}@/^0mm/|{y'} "b";"c"
  \ar@{->}@/_3mm/|{z'} "b";"c"
\end{xy}
\end{equation}

with the ideal of relations generated by
\[
x'y-y'x, y'z-z'y, z'x-x'z.
\]
The dual collection to $\O,\O(1),\O(2)$ is
$\Omega^2(2),\Omega^1(1),\O$. They map to simple modules
$S_2[-2],S_1[-1],S_0$ under $\RHom(T,-)$.

The roll-up helix algebra $B=\End(\pi^* T)$ is the path algebra of
the quiver $\cQ_\omega$

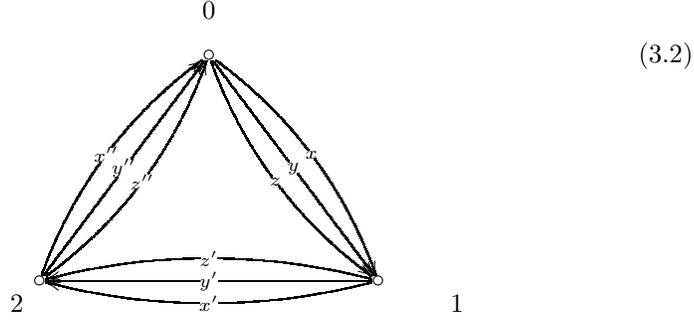
\begin{equation}
\begin{xy} <1.5mm,0mm>:
  (0,0)*{\circ}="a",(15,-20)*{\circ}="b",(-15,-20)*{\circ}="c",
  (0,4)*{0},(22,-22)*{1},(-17,-22)*{2}
  \ar@{->}@/^3mm/|{x} "a";"b"
  \ar@{->}@/^0mm/|{y} "a";"b"
  \ar@{->}@/_3mm/|{z} "a";"b"
  \ar@{->}@/^3mm/|{x'} "b";"c"
  \ar@{->}@/^0mm/|{y'} "b";"c"
  \ar@{->}@/_3mm/|{z'} "b";"c"
  \ar@{->}@/^3mm/|{x''} "c";"a"
  \ar@{->}@/^0mm/|{y''} "c";"a"
  \ar@{->}@/_3mm/|{z''} "c";"a"
\end{xy}
\end{equation}
with relations
\[
x'y-y'x, y'z-z'y, z'x-x'z;
\]
\[
x''y'-y''x', y''z'-z''y', z''x'-x''z';
\]
\[
xy''-yx'', yz''-zy'', zx''-xz''.
\]
\end{example}
\section{Cyclic completion of Yoneda algebra}\label{sec_cylic}

Two technical results are proved in this section.
\begin{enumerate}
\item[$\bullet$] First, we show the Yoneda algebra $L_\omega$ is the
cyclic completion of the Yoneda algebra $L$. This is the algebraic
counter part of the cotangent bundle construction.
\item[$\bullet$] Second, we show $\mu_k$ of
$L$ vanish when $k>\chi(X)$. Then by the cyclic completion
construction, so is true for $L_\omega$.
\end{enumerate}
Theorem \ref{Yoneda cyclic} was proved first by Aspinwall and
Fidkowski (Section 4.3 of \cite{Aspinwall}) and reproved in more
general form by Segal (Theorem 4.2 \cite{Segal}). For our own
convenience, we give a slightly different proof here. But the idea
is quite similar to \cite{Aspinwall} and \cite{Segal}.

These two results, together with the existence theorem of strong
exceptional collections (Theorem \ref{exceptional}) and Proposition
\ref{helix cond} guarantee the existence of global algebraic CS
functions. In fact, they provide a recipe to construct CS functions
starting from a strong exceptional collection satisfying Proposition
\ref{helix cond}.

\begin{definition}\label{cyclic completion}
Let $L=\bigoplus_{i=0}^d L^i$ be a finite dimensional $L_\infty$
algebra over $\bk$, with its $L_\infty$ products denoted by $\mu_k$.
Define $\bar{L}$ to be the graded vector space $L\oplus L[-d-1]$,
i.e. $\bar{L}^i=L^i\oplus (L^{d+1-i})^*$. Define the cyclic pairing
and $L_\infty$ products $\bar{\mu}_k: \wedge^k \bar{L} \to
\bar{L}[2-k]$ according to the following rules:
\begin{enumerate}
\item[(1)] define the bilinear form $\kappa$ on $\bar{L}$ by the natural pairing between
$L$ and $L^*$.
\item[(2)] if the inputs of $\bar{\mu}_k$ all belong to $L$ then define $\bar{\mu}_k=\mu_k$;
\item[(3)] if more than one input belong to $L^*$ then define
$\bar{\mu}_k=0$;
\item[(4)] if there is exactly one input $a_i^*\in L^*$ then define
$\bar{\mu}_k$ by
\[
\kappa(\bar{\mu}_k(a_1,\ldots,a_i^*,\ldots,a_k),b)=(-1)^{\epsilon}
\kappa(\mu_k(a_{i+1},\ldots,a_k,b,a_1,\ldots,a_{i-1}),a_i^*)\] for
arbitrary $b\in L$, where $\epsilon=|a_1|(|a_2|+\ldots+|b|)+\ldots+
|a_i^*|(|a_{i+1}|+\ldots+|b|)$;
\end{enumerate}
It is easy to check that $(\bar{L}, \bar{\mu}_k, \kappa)$ forms a
cyclic $L_\infty$ algebra. We call $\bar{L}$ the \emph{cyclic
completion} of $L$.
\end{definition}

We have defined the Yoneda algebras
$R=\Ext_A^\bullet(\bigoplus_{i=0}^n S_i,\bigoplus_{i=0}^n S_i)$ and
$R_\omega=\Ext^\bullet(\bigoplus_{i=0}^n
\iota_*S_i,\bigoplus_{i=0}^n \iota_*S_i)$ in previous section. Take
the associated $L_\infty$ algebras and denote them by $L$ and
$L_\omega$. The $d$ in Definition \ref{cyclic completion} equals $2$
since $X$ is a surface.

The following theorem will play a central role in this paper.
\begin{theorem}\label{Yoneda cyclic}
The Yoneda algebra $L_\omega$ is the cyclic completion of the Yoneda
algebra $L$.
\end{theorem}
\begin{proof}
This can be done in three steps. First, we need to verify that
$L_\omega$ and $\bar{L}$ coincide as graded vector spaces. Second,
we will show the pairing on $\bar{L}$ defined by $(1)$ of
\ref{cyclic completion} coincides with the Serre pairing on
$L_\omega$. Finally, we need to check that the $L_\infty$ products
on $L_\omega$ satisfy properties $(2)-(4)$ in Definition \ref{cyclic
completion}.

Given an object $E\in \D^b(mod-B)\simeq \D^b(Y)$ that is scheme
theoretically support on $X$, one can view $E$ as a complex of
finitely generated $A$-modules.

There is a projective $A$ resolution $P^\bullet$ for $E$:
\[
\xymatrix{P^\bullet\ar[r] &E\ar[r] &0}
\] such that all $P^i$ are direct sum of copies of
$E_0,\ldots,E_{n}$.

Because $Y$ is the total space of canonical bundle over $X$, there
is a tautological short exact sequence of sheaves
\[
\xymatrix{0\ar[r] &\pi^*(\omega_X^{-1})\ar[r] &\O_Y\ar[r]
&\O_X\ar[r] &0}.
\]
Tensor it by $\pi^*E$ we obtain
\[
\xymatrix{0\ar[r] &\pi^*E(\omega_X^{-1})\ar[r] &\pi^*E\ar[r]
&\iota_*E\ar[r] &0}.
\]
Since $\pi^*$ preserves the projective modules, by replacing $E$
with $P^\bullet$ we obtain a projective $B$ resolutions of $\iota_*
E$ as total complex of the following double complex
\begin{equation}\label{resolution}
\xymatrix{ \ldots\ar[r]&\pi^*P^{-2}\ar[r]&\pi^*P^{-1}\ar[r]&\pi^*P^0\ar[r]&0\\
\ldots\ar[r]&\pi^*P^{-2}\otimes\pi^*\omega_X^{-1}\ar[u]\ar[r]&\pi^*P^{-1}\otimes\pi^*\omega_X^{-1}
\ar[u]\ar[r]&\pi^*P^0\otimes\pi^*\omega_X^{-1}\ar[u]\ar[r]&0}
\end{equation}
We denote this resolution of $\iota_* E$ by $P^\bullet_\omega$.

As graded vector space, $L_\omega$ is computed as cohomology of
$\Hom^\bullet_Y(P^\bullet_\omega, \iota_* E)$. Because
$P^\bullet_\omega$ is the total complex of the above double complex,
$\Hom^\bullet_Y(P^\bullet_\omega, \iota_* E)$ is quasi-isomorphic
with the total complex of the following double complex:
\[
\xymatrix{ \ldots\Hom(\pi^*P^{-2},\iota_*
E)\ar[d]&\Hom(\pi^*P^{-1},\iota_* E)\ar[l]\ar[d]
&\Hom(\pi^*P^0,\iota_* E)\ar[l]\ar[d]&0\ar[l]\\
\ldots\Hom(\pi^*P^{-2}\otimes\pi^*\omega_X^{-1},\iota_*
E)&\Hom(\pi^*P^{-1}\otimes\pi^*\omega_X^{-1},\iota_* E)
\ar[l]&\Hom(\pi^*P^0\otimes\pi^*\omega_X^{-1},\iota_*
E)\ar[l]&0\ar[l]}
\]
The spectral sequence associated to this double complex degenerates
at $E_1$ page. Using adjunction together with Serre duality, we
obtain
\\
$
\Hom(\iota_*E,\iota_*E)=\Hom_X(E,E);\\
\Ext^1(\iota_*E,\iota_*E)=\Ext^1_X(E,E)\oplus\Ext^2_X(E,E)^*;\\
\Ext^2(\iota_*E,\iota_*E)=\Ext^2_X(E,E)\oplus\Ext^1_X(E,E)^*;\\
\Ext^3(\iota_*E,\iota_*E)=\Hom_X(E,E)^*. $

The above fact holds for any object $E$ support scheme theoretically
on X. We are particularly interested in the case when $E$ is
$\oplus_{i=0}^n F_{-i}[i]$, i.e. the direct sum of simple objects in
$mod-A$. This identifies $L_\omega$ and $\bar{L}$ as graded vector
spaces since both will equal to $L \oplus L[-3]^*$.

To verify property $(1)$, we need to write down a bilinear pairing
$\kappa$ on $\Hom^\bullet(P^\bullet_\omega,P^\bullet_\omega)$ such
that its restriction on cohomology gives the obvious duality between
$L$ and $L^*$. By adjunction,
$\Hom^3(P^\bullet_\omega,P^\bullet_\omega)$ has a direct summand
$\Hom^2(\pi^*P^\bullet\otimes\omega_X^{-1},\pi^*P^\bullet)$, which
is isomorphic to $\Hom^2_X(P^\bullet,P^\bullet\otimes(\oplus_{k\leq
1} \omega_X^k))$. It contains the finite dimensional graded piece
$\Hom^2_X(P^\bullet,P^\bullet\otimes\omega)$, which has a trace map
to $H^2(X,\omega_X)\simeq\CC$. Given any two elements $x$ and $y$ in
$\Hom^\bullet(P^\bullet_\omega,P^\bullet_\omega)$, we define the
bilinear pairing $\kappa(x,y)$ to be the projection of $x\circ y$ to
the graded piece $\Hom^2_X(P^\bullet,P^\bullet\otimes\omega)$
following by the trace map. Clearly, the restriction of $\kappa$ on
cohomology satisfies property $(1)$.

Then we need to verify properties $(2)$ to $(4)$ for $L_\omega$. For
the dimension reason, it suffices to check the case when all the
inputs of the $L_\infty$ products $\mu_k$ are in $L^1_\omega$. Since
$L_\omega$ is constructed as cohomology of
$\Hom^\bullet(P^\bullet_\omega,P^\bullet_\omega)$, the element in
$L^1_\omega$ can be represented by either the vertical or horizontal
arrows in diagram \ref{resolution}. More specifically, a class in
$\Ext^1_X(E,E)$ is represented by a horizontal arrow and a class in
$\Ext^2_X(E,E)^*$ is represented by a vertical arrow. Then property
$(1)$ follows immediately since the rows of the double complex is
simply the pullback of $P^\bullet$ (up to $\otimes \omega_X^{-1}$),
which is the projective resolution of $E$.

If we write $\Ext^2(E,E)^*$ as $\Ext^0(E,E\otimes \omega_X)$, then
we will see that
\[\bar{\mu}_2: \Ext^1(E,E)\bigotimes
\Ext^0(E,E\otimes\omega_X)\to
\Ext^1(E,E\otimes\omega_X)\simeq\Ext^1(E,E)^*\] is the only non-zero
term that can involve $\Ext^2(E,E)^*$. For example, if both inputs
of $\mu_2$ belong to $\Ext^0(E,E\otimes\omega_X)$, the output would
be $\Ext^0(E,E\otimes \omega^2_X)$, which is not in $L^2_\omega$.
Similarly, this argument shows that any nonzero term of $\mu_k$ of
$L_\omega$ can at most involve one $\Ext^2(E,E)^*$ term. This proves
Property $(2)$.

Property $(3)$ is essentially the cyclic symmetry of $\mu_k$. Since
the $\kappa$ on cohomology is a restriction of a bilinear form
(denote also by $\kappa$) on the dga
$\Hom^\bullet(P^\bullet_\omega,P^\bullet_\omega)$ with differential
$d$. Property $(3)$ will follow from the following cyclic symmetry
properties on $\Hom^\bullet(P^\bullet_\omega,P^\bullet_\omega)$. For
arbitrary elements $x$, $y$ and $z$:
\begin{enumerate}
\item[$\diamond$] $\kappa(x,y)=\pm \kappa(y,x)$
\item[$\diamond$] $\kappa(dx,y)=\pm \kappa(dy,x)$;
\item[$\diamond$] $\kappa(x\circ y,z)=\pm \kappa(y\circ z,x)$.
\end{enumerate}
The first property is clear since the commutator is trace-free. The
trace map will factor through the morphism
$\Hom^2(P^\bullet,P^\bullet\otimes\omega)\to
L^3_\omega=\Ext^2(E,E\otimes\omega)\simeq \Hom(E,E)^*$. Therefore
trace of a coboundary will be zero. Then the second property will
follow from the Leibniz rule. The third property follows from the
first and associativity of product.

\end{proof}

\begin{remark}({\bf{The geometric meaning of cyclic
completion}})\label{geometry_cyclic} Recall from section
\ref{Linfty-dg mfd} that the completed symmetric algebra
$\widehat{Sym\ L[1]^*}$ (we omit $\tau^{>0}$ for simplicity) can be
interpreted as $\O_X^\bullet$ of the dg-manifold $\cX$ defined from
$L$. This dg-algebra can be viewed as the structure sheaf of the
graded linear manifold $M=L[1]$.

The odd cotangent bundle of the graded manifold $M$, denoted as
$T^*[-1]M$ is defined to be the graded manifold $L[1]\oplus
(L[1]^*[-1])$. As graded vector spaces, $T^*[-1]M$ is same as
$L_\omega[1]$. Then $\cO_{T^*[-1]M}$ coincide with $\widehat{Sym \
L_\omega[1]^*}$ as graded algebras. The $L_\infty$ products
$\bar{\mu}_k$ defines a derivation on $\cO_{T^*[-1]M}$ and the
cyclic pairing $\kappa$ defines an odd two form on $T^*[-1]M$. In
fact this process is functorial. Hence, passing to the cyclic
completion of a $L_\infty$ algebra is an algebraic counterpart for
taking odd cotangent bundle of a dg-manifold.
\end{remark}

The $L_\infty$ (or $A_\infty$) structure of the Yoneda algebra $L$
has been studied for a long time in representation theory of finite
dimensional algebra. The following boundness theorem turns out to be
very important for the purpose of this paper.

\begin{theorem}\label{boundness of mu_k}
The $L_\infty$ products (higher brackets) $\mu_k$ on $L$ vanish when
$k> \chi(X)$.
\end{theorem}
\begin{proof}
Let $A$ be a finite dimensional algebra and $\{S_i\}$ be the
collection of simple $A$-modules. It is well known that the Yoneda
algebra $R=\Ext^\bullet_A(\oplus S_i,\oplus S_i)$ controls the
deformation of $A$, because it is quasi-isomorphic (as $A_\infty$
algebras) to the Hochschild homology of $\D^b(mod-A)$ (See
\cite{Hoch}). If $A$ is presented as a path algebra of a quiver with
relations, then the $A_\infty$ products $m_k$ on $R$ can be
interpreted as relations of the path algebra (See section 7.8 of
\cite{Keller}).

Since in our situation the quiver is constructed from a strong
exceptional collection of line bundles on $X$ (recall the
construction in section 3), the elements in the path algebra $A$
carry an extra grading given by the ordering on the strong
exceptional collection. The $A_\infty$ products preserve this extra
grading. Therefore the length of strong exceptional collection,
which is equal to the Euler characteristic $\chi(X)$, gives an upper
bound for number of non-vanishing $m_k$. This is clear intuitively
since on a directed quiver generating by length $4$ strong
exceptional collection, there cannot be a relation involving length
$5$ paths.

Finally, we pass from $A_\infty$ algebra to $L_\infty$ algebra.
Since $L$ is the antisymmetrization of $R$, $\mu_k=0$ when $m_k=0$.
\end{proof}

The interpretation of Yoneda algebra as Hochschild homology also
gives the following corollary.
\begin{corollary}
The Yoneda algebra $R$ constructed from different strong exceptional
collections on $X$ are quasi-isomorphic as $A_\infty$ algebras, i.e.
it is an invariant of the derived category $\D^b(X)$.
\end{corollary}
\begin{proof}
See \cite{Hoch}.
\end{proof}

\section{Moduli space and Chern-Simons function}\label{sec_CS}
We fix the ground field $\bk=\CC$. Let $\Gamma$ be the Grothendieck
group of $\D_\omega$. By the derived Morita equivalence, $\Gamma$
also equals to the Grothendieck group of derived category of
nilpotent representations of $\cQ_\omega$. It is a free abelian
group of rank $n+1$ generated by the collection of simple modules
$\iota_*S_0,\ldots,\iota_*S_n$. If we fix these simple modules as a
$\ZZ$-basis of $\Gamma$, every effective class can be written as a
vector $\bd=(d_0,\ldots,d_n)$ with nonnegative entries. We call such
a choice of $\bd$ a dimension vector.

\begin{theorem} \label{moduli of sheaves}
Let $X$ be a toric Fano stack of dimension two and $Y$ be its total
space of canonical bundle. Pick a strong exceptional collection
constructed by Theorem \cite{BorHua} and denote the corresponding
quiver of $Y$ by $\cQ_\omega$. Let $\mathfrak{M}_\gamma$ be a
bounded family of sheaves on $Y$ support on $X$ with class
$\gamma\in\Gamma$. There exists a dimension vector $\bd$ and an open
immersion of Artin stacks from $\mathfrak{M}_\gamma$ to the quotient
stack $[MC(L_{\omega,\bd})/G_\bd]$, where $MC(L_{\omega,\bd})$ is
the space of nilpotent representations of $\cQ_\omega$ with
dimension vector $\bd$ and $G_\bd$ (defined later in this section)
is the gauge group acting by changing of basis.
\end{theorem}

\begin{theorem}\label{Critical set}
Given a class $\gamma\in \Gamma$, a bounded family of sheaves on $Y$
support on $X$ with class $\gamma$ is the critical set of an
algebraic function $f_\bd$. We call such a function a CS function.
\end{theorem}

The infinitesimal deformation of representations is controlled by
the following $L_\infty$ algebra.

Fix a dimension vector $\bd$, define
\[
L_{\bd}:=\Ext^\bullet(\bigoplus_{i=0}^{n} S_i\otimes
V_i,\bigoplus_{i=0}^{n} S_i\otimes V_i)
\] and
\[
L_{\omega,\bd}:=\Ext^\bullet(\bigoplus_{i=0}^{n} \iota_*S_i\otimes
V_i,\bigoplus_{i=0}^{n} \iota_*S_i\otimes V_i)
\]
where $V_i$ are vector spaces of dimension $d_i$. They are
generalizations of the Yoneda algebras. If we take
$\bd=(1,\ldots,1)$ we obtain the Yoneda algebras. All the results in
Section \ref{sec_cylic} clearly generalize to $L_\bd$ and
$L_{\omega,\bd}$.

The space $L^1_\bd$ (resp. $L^1_{\omega,\bd}$) can be identified
with the space $\bigoplus_{a\in \cQ_1} \Hom(V_i,V_j)$ (resp. $a\in
{\cQ_\omega}_1$) of matrices summing over all the arrows. It carries
a natural bi-grading by the source and target of arrow. The space
$L^0_\bd$ (resp. $L^0_{\omega,\bd}$) can be identified with the
space $\bigoplus_{i\in \cQ_0} \End(V_i)$, which is the Lie algebra
associate to the group $\prod_{i\in \cQ_0} GL(V_i)$. We denote this
group by $G_\bd$ for simplicity. It acts on $L_\bd$ (resp.
$L_{\omega,\bd}$) by conjugation.

\begin{lemma}\label{MC-relations}
The elements of $MC(L_\bd)$ (resp. $MC(L_{\omega,\bd})$) are in one
to one correspondence with representations of $\cQ$ (resp.
$\cQ_\omega$) of dimension vector $\bd$. Two representations are
isomorphic if and only if they belong to the same orbits of $G_\bd$.
\end{lemma}
\begin{proof}
See Section 7.8 of \cite{Keller} or Proposition 3.8 \cite{Segal}.
\end{proof}

The $L_\infty$ algebra $L$ (resp. $L_{\omega,\bd}$) controls the
infinitesimal deformation of representations in the following sense.
Let $M$ be an $A$ module (resp. $B$ module) with dimension vector
$\bd$. We denote its corresponding Maurer-Cartan element by $x$. The
homology groups $H^i(L_\bd,\delta^x)$ (resp.
$H^i(L_{\omega,\bd},\delta^x)$) are isomorphic to $\Ext^i_A(M,M)$
(resp. $\Ext^i_B(M,M)$). In general, $L_\bd$ is just the formal
tangent space at the point $\bigoplus S_i\otimes V_i$. However, in
our situation because of the boundness of $\mu_k$ (Theorem
\ref{boundness of mu_k}), the Maurer-Cartan equation actually
converges. Therefore the moduli space can be constructed globally as
mentioned in the previous Lemma.

\paragraph{Proof of Theorem \ref{moduli of sheaves}}
\begin{proof}
Given Lemma \ref{MC-relations}, it suffices to show the existence of
an open immersion from $\mathfrak{M}_\gamma$ to
$[MC(L_{\omega,\bd})/G_\bd]$.

First we need to construct a monomorphism of stacks. Let's pick an
ample line bundle $L$ on $X$. If $T$ is a tilting bundle on $X$ then
$T\otimes L^{-N}$ is again a tilting bundle for any integer $N$.
Therefore the functor $\RHom(\pi^*(T\otimes L^{-N}),-)$ induces an
equivalence from $\D^b(Y)$ to $\D^b(mod-B)$. Because $T$ is direct
sum of line bundles, we can choose $N\gg 0$ such that for any sheaf
$\E\in\mathfrak{M}_\gamma$, $\RHom(\pi^*(T\otimes L^{-N}),\E)$ is
concentrated on degree zero, i.e. is a module over $B$. Let $\bd$ be
its dimension vector, which depends on both $\gamma$ and $N$. Then
we obtain a morphism between stacks. Because of the derived Morita
equivalence, this is clearly an injection.

Next we need to argue this morphism is $\acute{e}$tale. Let $A'\to
A\to \CC$ be a small extension of pointed $\CC$-algebras. Let $T =
Spec A$ and $T' = Spec A'$. Consider a $2$-commutative diagram
\[
\xymatrix{ T\ar[r]\ar@{^(->}[d] &
\mathfrak{M}_\gamma\ar@{^(->}[d]^{\RHom(\pi^*(T\otimes
L^{-N}),-)} \\
T'\ar[r]\ar@{-->}[ru]&[MC(L_{\omega,\bd})/G_\bd]}
\]
of solid arrows. We have to prove that the dotted arrow exists,
uniquely, up to a unique $2$-isomorphism. This follows from standard
deformation theory. We need that $\RHom(\pi^*(T\otimes L^{-N}),-)$
induces a bijection on deformation spaces and an injection on
obstruction spaces (associated to the above diagram). They follow
immediately for the equivalence between $\D^b(Y)$ and $\D^b(mod-B)$.
In fact, all the obstruction groups are isomorphic.
\end{proof}

\paragraph{Proof of Theorem \ref{Critical set}}
\begin{proof}
As we have seen in Definition \ref{hamiltonian}, there is always a
formal function
\[
f_\bd(z) = \sum_{k=1}^\infty \frac{(-1)^{\frac{k(k+1)}{2}}}{(k+1)!}
\kappa( \bar{\mu}_k(z,\ldots,z), z)
\]
associated to the cyclic $L_\infty$ algebra $L_{\omega,\bd}$ where
$z\in L_{\omega,\bd}^1$. The critical set of $f_\bd$ coincides with
$MC(L_{\omega,\bd})$.

By the boundness theorem \ref{boundness of mu_k}, such formal
function is in fact a polynomial function of degree at most
$\chi(X)$. Therefore, $MC(L_{\omega,\bd})$ as a subvariety of
$L^1_{\omega,\bd}$ is the critical scheme of $f_\bd$. Since the
$G_\bd$ action is induced from the action of the Lie subalgebra
$L^0_{\omega,\bd}$, $f_\bd$ is clearly invariant under this action.

By Theorem \ref{moduli of sheaves}, $\mathfrak{M}_\gamma$ is an open
substack of $[MC(L_{\omega,\bd})/G_\bd]$ for appropriate choice of
$\bd$. The theorem follows since we can restrict the function
$f_\bd$.
\end{proof}

\begin{remark}\label{inner product}
Recall that by Theorem \ref{Yoneda cyclic},  $L^1_{\omega,\bd}$
decomposes into $L^1_\bd\oplus (L^2_\bd)^*$. The CS function $f_\bd$
has a nice property respect to this decomposition.

If we write the cyclic pairing $\kappa(x,y)$ as $tr(x\circ y)$, the
function $f_\bd$ can be written as the trace of cyclic invariant
polynomial of matrices. Theorem \ref{cyclic completion} tells us
that the variables in $(L^2_\bd)^*$ appear exactly once (of degree
one) in all the monomials. This means we can always write $f_\bd$ as
an inner product of a polynomials of elements in $L^1_\bd$ and
elements of $(L^2_\bd)^*$. This property plays a central role in
Section \ref{sec_motive reduction}.

\end{remark}

As a summary of Section \ref{sec_cylic} and \ref{sec_CS}, we give an
algorithm to compute CS functions on local toric Fano surfaces.
\begin{enumerate}
\item[{\bf{STEP} 1}] Choose a strong exceptional collection of line bundles on $X$. By
results in Section \ref{sec_morita}, this completely determines the
quiver $\cQ$ together with its relations.
\item[{\bf{STEP} 2}] Compute the $A_\infty$ structures on the Yoneda
algebra $R$ using the correspondence between $m_k$ and relations to
the $\cQ$.
\item[{\bf{STEP} 3}] Apply Theorem \ref{Yoneda cyclic} to compute
$\bar{m}_k$ of $R_\omega$.
\item[{\bf{STEP} 4}] Plug in specific dimension vector $\bd$, anti-commutatize $R_{\omega,\bd}$
to $L_{\omega,\bd}$ and apply Definition \ref{hamiltonian} to
compute $f_\bd$.
\end{enumerate}

\section{Examples of CS functions}\label{sec_examples}
In these section, we discuss some examples of CS functions.
\subsection{$\CC^3$}
The easiest example of a Calabi-Yau 3-fold is the three dimensional
affine space. Rigorously speaking, it is not a local surface but
still the CS function can be computed using the same philosophy.

Let $B$ be the polynomial algebra with three variables. The category
$Coh (\CC^3)$ equals $mod-B$. Consider the quiver $\cQ_\omega$.
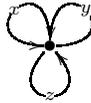
\begin{equation}
\begin{xy} <1.5mm,0mm>:
  (10,-10)*{\bullet}="a"
  \ar@(u,l)|{x} "a";"a"
  \ar@(r,u)|{y} "a";"a"
  \ar@(ld,rd)|{z} "a";"a"
\end{xy}
\end{equation}
with relations $xy-yx,yz-zy,zx-xz$. Its path algebra is equal to
$B$.

Given a positive integer $n$, let $L_{\omega,n}$ be the Yoneda
algebra $\Ext^\bullet_{\CC^3}(\O_{\{0\}},\O_{\{0\}})\otimes
\frak{gl}_n$. Since the only non-vanishing product is $\bar{\mu}_2$,
$L_{\omega,n}$ is a graded Lie algebra. Let $A,B,C$ be $n$ by $n$
matrices associated to $x,y,z$. The CS function $f_n$ is equal to
$tr((AB-BA)C)$.

The Morita equivalence in this case is the classical Koszul duality
between symmetric and exterior algebras
\[
\D^b(coh(V))=D^b(mod-\wedge^\bullet (V)).
\]

The quiver $\cQ_\omega$ gives combinatoric description for both
$\CC^3$ and cotangent bundle of three dimensional torus. The first
is clear since the path algebra of $\cQ_\omega$ is the algebra of
functions on $\CC^3$. For the second, we can think of the quiver as
the 1-skeleton of $T^3$ and the relations as the the gluing
conditions of two cells.

The stack $[crit f_\bd/G_\bd]$ is related to two interesting moduli
spaces. The first one \footnote{One can modify the construction
slightly to include the Hilbert scheme of points (See \cite{BBS}.)}
is the moduli space of length $n$ sheaves on $\CC^3$ and the second
one \footnote{One need to put a stability condition to make it hold
rigorously.} is the moduli space of flat $GL_n$ vector bundles on
$T^3$. These two moduli spaces are related by Homological Mirror
Symmetry.

\subsection{$\omega_\PP^2$}
Following from calculations done in Example \ref{P2}, the CS
function of local projective plane is given by
\[
tr(C''(A'B-B'A)+A''(B'C-C'B)+B''(C'A-A'C))
\]
where $A,B,C,A',B',C',A'',B'',C''$ are matrices associated to arrows
$x,y,z,x',y',z',x'',y'',z''$.

\subsection{$\omega_{\PP(1:3:1)}$ and $\omega_{\PP(2:1:2)}$}
In this subsection, we will compute the CS functions of
$\omega_{\PP(1:3:1)}$ and $\omega_{\PP(2:1:2)}$. These two
Calabi-Yau 3-folds are K-equivalent. There is some interesting
symmetry between these two CS functions.

For simplicity, we denote $\PP(1:3:1)$ and $\PP(2:1:2)$ by $X_1$ and
$X_2$ respectively. The stacky fan $\Sigma_1$ of $X_1$ has rays:
$(0,1), (1,-1), (-1,-2)$. The stacky fan $\Sigma_2$ of $X_1$ has
rays: $(0,2), (1,0), (-1,-1)$. Denote their canonical bundles by
$Y_1$ and $Y_2$.

The Picard groups of $X_1$ and $X_2$ both equal to $\ZZ$. We denote
the positive generator by $\O(1)$. On $X_1$, $\O(1)$ can be written
as $\O(D_2)$ with $D_2$ being the toric invariant divisor for
$(1,-1)$. On $X_2$, $\O(1)$ can be written as $\O(D_1)$ with $D_1$
being the toric invariant divisor for $(0,2)$. For both $\D^b(X_1)$
and $\D^b(X_2)$, $\O,\O(1),\O(2),\O(3),\O(4)$ form a full strong
exceptional collection. The quivers associated to these two
collections are denoted by $\cQ_1$ and $\cQ_2$. The sets of vertices
$\{0,1,2,3,4\}$ correspond to $\O,\O(1),\O(2),\O(3),\O(4)$.

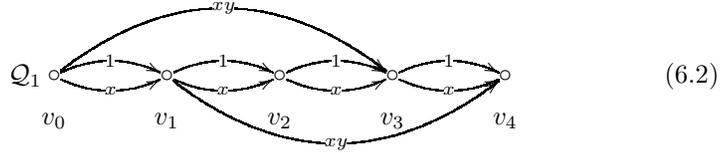
\begin{equation}{\cQ_1}
\begin{xy} <1.5mm,0mm>:
  (0,0)*{\circ}="a",(10,0)*{\circ}="b",(20,0)*{\circ}="c",(30,0)*{\circ}="d",(40,0)*{\circ}="e",
  (0,-4)*{v_0},(10,-4)*{v_1},(20,-4)*{v_2},(30,-4)*{v_3},(40,-4)*{v_4}
  \ar@{->}@/^2mm/|{1} "a";"b"
  \ar@{->}@/_2mm/|{x} "a";"b"
  \ar@{->}@/^2mm/|{1} "b";"c"
  \ar@{->}@/_2mm/|{x} "b";"c"
  \ar@{->}@/^2mm/|{1} "c";"d"
  \ar@{->}@/_2mm/|{x} "c";"d"
  \ar@{->}@/^2mm/|{1} "d";"e"
  \ar@{->}@/_2mm/|{x} "d";"e"
  \ar@{->}@/^9mm/|{xy} "a";"d"
  \ar@{->}@/_9mm/|{xy} "b";"e"
\end{xy}
\end{equation}

\begin{equation}{\cQ_2}
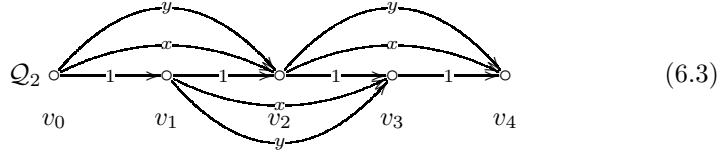
\begin{xy} <1.5mm,0mm>:
  (0,0)*{\circ}="a",(10,0)*{\circ}="b",(20,0)*{\circ}="c",(30,0)*{\circ}="d",(40,0)*{\circ}="e",
  (0,-4)*{v_0},(10,-4)*{v_1},(20,-4)*{v_2},(30,-4)*{v_3},(40,-4)*{v_4}
  \ar@{->}|{1} "a";"b"
  \ar@{->}|{1} "b";"c"
  \ar@{->}|{1} "c";"d"
  \ar@{->}|{1} "d";"e"
  \ar@{->}@/^4mm/|{x} "a";"c"
  \ar@{->}@/^4mm/|{x} "c";"e"
  \ar@{->}@/_4mm/|{x} "b";"d"
  \ar@{->}@/^9mm/|{y} "a";"c"
  \ar@{->}@/^9mm/|{y} "c";"e"
  \ar@{->}@/_9mm/|{y} "b";"d"
\end{xy}
\end{equation}
Notice that $\Sigma_1$ and $\Sigma_2$ are related by shift of
origin. This shift changes the stack completely. But surprisingly,
the full strong exceptional collections on $X_{\Sigma_1}$ and
$X_{\Sigma_2}$ are related (\cite{BorHua}).

Following from the algorithm in the end of last section, the CS
function for $\omega_{\PP(1:3:1)}$ is
\begin{equation}
\begin{split}
f=tr(R_{20}(B_{12}A_{01}-A_{12}B_{01})+R_{31}(B_{23}A_{12}-A_{23}B_{12})\\
+R_{42}(B_{34}A_{23}-A_{34}B_{23})+R_{40}(A_{34}C_{03}-C_{14}A_{01})\\
+S_{40}(B_{34}C_{03}-C_{14}B_{01})).
\end{split}
\end{equation}

The CS function for $\omega_{\PP(2:1:2)}$ is
\begin{equation}
\begin{split}
f=tr(R_{30}(A_{23}B_{02}-B_{13}A_{01})+R_{41}(B_{24}A_{12}-A_{34}B_{13})\\
+S_{30}(A_{23}C_{02}-C_{13}A_{01})+S_{41}(C_{24}A_{12}-A_{34}C_{13})\\
+R_{40}(B_{24}C_{02}-C_{24}B_{02})).
\end{split}
\end{equation}

\subsection{$\PP^2$ blow up at one point}\label{dP1}
The first example which involve $\mu_k$ terms with $k>2$ is the
local DelPezzo surface of degree one. It is first computed by
Aspinwall and Fidkowski in \cite{Aspinwall}.

Let $X$ be the blow up of $\PP^2$ at one point. Denote the pull back
of hyperplane by $H$ and the exceptional divisor by $E$. The derived
category $\D^b(X)$ has a strong exceptional collection
$\O,\O(H),\O(2H-E),\O(2H)$. The corresponding quiver is:

\begin{equation}{\cQ_3}
\begin{xy} <1.5mm,0mm>:
  (0,0)*{\circ}="a",(15,0)*{\circ}="b",(30,0)*{\circ}="c",(45,0)*{\circ}="d",
  (0,-8)*{\O},(15,-8)*{\O(H)},(30,-8)*{\O(2H-E)},(45,-8)*{\O(2H)}
  \ar@{->}@/^2mm/|{d_0} "a";"b"
  \ar@{->}@/^0mm/|{d_1} "a";"b"
  \ar@{->}@/_2mm/|{d_2} "a";"b"
  \ar@{->}@/^1mm/|{b_0} "b";"c"
  \ar@{->}@/_1mm/|{b_1} "b";"c"
  \ar@{->}@/^0mm/|{a} "c";"d"
  \ar@{->}@/^5mm/|{c} "b";"d"
  \ar@{.>}@/_10mm/|{s_0} "d";"a"
  \ar@{.>}@/_8mm/|{s_1} "d";"a"
  \ar@{.>}@/^5mm/|{r_0} "c";"a"
\end{xy}
\end{equation}
The graded piece $L^2$ of the Yoneda algebra has dimension three. We
denote the basis by $r_0, s_0, s_1$.  If we denote the matrices
associated to each arrow by Capital letters, then the CS function is
given by
\[
f=tr(R_0(B_0D_1-B_1D_0) + S_0(AB_0D_2-CD_0) + S_1(AB_1D_2-CD_1)).
\]

\section{Integrality of Generalized DT invariants}\label{sec_Integrality}
In this section, we give the first geometric application of CS
functions. The main result is Theorem \ref{smoothness}, where we
show that the $L_\infty$ products vanish at semistable points of the
moduli space of sheaves of local surfaces. As a consequence, the
generalized Donaldson-Thomas invariants defined by Joyce and Song
\cite{JS} are integral on local surfaces.

We only consider sheaves on $Y$ that belong to the category
$\D_\omega$, i.e. set theoretically support on $X$. Furthermore, we
assume they are support on dimension bigger than zero. The
integrality of the zero dimensional sheaves has been proved in
section 6.3 in \cite{JS}.

Let $L$ be an ample line bundle on $X$. The Hilbert polynomial of a
coherent sheaf $\E$ on $Y$ is defined to be $\chi(\E\otimes
\pi^*{L^k})$ for $k\gg0$. The slope of $\E$, denoted by $\mu(\E)$ is
defined to be the quotient of the second nonzero coefficient of its
Hilbert polynomial by the first. We will adopt the notation of Joyce
and Song \cite{JS}. A sheaf $\E$ is called $\tau$-stable
(semi-stable) if for any proper subsheaf $\F$, $\mu(\F)<\mu(\E)$
($\leq$). The moduli space of $\tau$ semi-stable sheaves on $Y$ with
class $\gamma\in \Gamma$ is denoted by
$\mathfrak{M}^\tau(Y,\gamma)$.

\begin{lemma}\label{polystable is pushforward}
If $\E$ is a $\tau$-stable sheaf on $Y$, then $\E$ is supported on
$X$ scheme theoretically.
\end{lemma}
\begin{proof}
First, stability implies that $\E$ must be pure. Second, the support
of $\E$ must be reduced. Suppose this is not the case. Call the
scheme theoretical support of $\E$ by $Z$. Then
$\Hom_X(\cO_Z,\cO_Z)\subset\Hom_X(\E,\E)$. Because $Z$ is
non-reduced, $\Hom_X(\cO_Z,\cO_Z)$ has $\CC$-dimension strictly
bigger than one. On the other hand, since X has negative normal
bundle inside $Y$, every element in $\Hom_X(\E,\E)$ extends to
$\Hom_Y(\E,\E)$. Therefore $dim_\CC \Hom_Y(\E,\E)>1$, which violates
the stability of $\E$.
\end{proof}

\begin{lemma}\label{polystable no Ext2}
Let $\E_1$ and $\E_2$ be two $\tau$ semi-stable sheaves on $X$ such
that $\mu(\E_1)=\mu(E_2)$. Then $\Ext^2(\E_1,\E_2)=0$.
\end{lemma}
\begin{proof}
By Serre duality, $\Ext^2_X(\E_1,\E_2)=\Hom_X(\E_2,\E_1\otimes
\omega_X)^*$. Because $\omega_X^{-1}$ is ample and $\E_1$, $\E_2$
have dimension bigger than zero, $\mu(\E_1\otimes
\omega_X)<\mu(\E_1)=\mu(\E_2)$. Hence $\Ext^2(\E_1,\E_2)$ vanishes
by stability.
\end{proof}

Lemma \ref{polystable is pushforward} doesn't hold for semistable
sheaves. For example, if we take a proper but non-reduced curve in
$Y$. Its structure sheaf can be semi-stable but not stable.

\begin{lemma}\label{pullback of semistable}
Let $\E$ be a $\tau$ semi-stable sheaf on $Y$. Then the restriction
$\E|_X$ is a semi-stable sheaf on $X$.
\end{lemma}
\begin{proof}
Because $\E$ is set theoretically supported on $X$, it can be
written as consequent extensions of stable sheaves on $X$ with the
same slope (Jordan-Holder filtration). Furthermore, the natural
morphism $\E|_X\to\E$ is always an injection of sheaves. Since
$\mu(\E|_X)=\mu(\E)$, any subsheaf that destabilizes $\E|_X$ will
destabilize $\E$ as well.
\end{proof}

\begin{theorem}\label{smoothness}
The $L_\infty$ products $\bar{\mu}_k$ of $L_\omega$ vanish at
semi-stable points.
\end{theorem}
\begin{proof}
Let $\E$ be a $\tau$ semistable sheaf on $Y$. It follows from
Theorem \ref{moduli of sheaves} that we can define a cyclic
$L_\infty$ algebra $L_\omega$ such that $\E$ is mapped to a
Maurer-Cartan element $\bar{x}$ and furthermore $\Ext^i_Y(\E,\E)$
coincide with $H^i(L_\omega,\delta^{\bar{x}})$. The $L_\infty$
products $\bar{\mu}_k$ uniquely defines $L_\infty$ products on
$H^\bullet(L_\omega,\delta^{\bar{x}})$ up to homotopy. We say
$\bar{\mu_k}$ vanish at $x$ if they vanish after passing to
$H^\bullet(L_\omega,\delta^{\bar{x}})$.

A MC element $\bar{x}$ of $L_\omega$ decomposes into $(x,\epsilon)$
respect to the decomposition $L^1_\omega=L^1\oplus (L^2)^*$. It
follows from Theorem \ref{Yoneda cyclic} that $x$ is a MC element of
$L$. The cohomology $H^\bullet(L_\omega,\delta^{\bar{x}})$ can be
computed as the cohomology of the total complex of
\[
\xymatrix{{L^2}^*\ar[r]&{L^1}^*\ar[r]&{L^0}^*\\
L^0\ar[r]\ar[u]&L^1\ar[r]\ar[u]&L^2\ar[u]}
\] where the horizontal differential is $\delta^x$ and the vertical
differential is induced by $\epsilon$.

If $\bar{x}$ is the image of a sheaf of the form $\iota_*\E$ for
some sheaf $\E$ on $X$ then $\bar{x}=(x,0)$. The associated spectral
sequence will degenerate at $E_1$ page.

If $\epsilon\neq 0$, then we need to pass to the $E_2$ page of:
\[
\xymatrix{{H^2(L,\delta^x)}^*\ar[r]^0&{H^1(L,\delta^x)}^*\ar[r]&{H^0(L,\delta^x)}^*\\
H^0(L,\delta^x)\ar[r]\ar[u]^\epsilon&H^1(L,\delta^x)\ar[r]\ar[u]^\epsilon&H^2(L,\delta^x)\ar[u]^\epsilon}
\]

Therefore, $H^1(L_\omega,\delta^{\bar{x}})$ and
$H^2(L_\omega,\delta^{\bar{x}})$ are equal to the kernel and
cokernel of the map $\xymatrix{H^1(L,\delta^x)\ar[r]^\epsilon
&{H^1(L, \delta^x)}^*}$. The MC element $(x,0)$ is exactly the one
corresponding to $\E|_X$. So
$H^i(L,\delta^x)=\Ext_X^i(\E|_X,\E|_X)$. Now by Lemma \ref{pullback
of semistable}, $H^2(L,\delta^x)$ vanishes.

By similar argument of Theorem \ref{Yoneda cyclic}, if the inputs of
$\bar{\mu}_k$ belongs to $H^1(L,\delta^x)$ then the output must be
inside $H^2(L,\delta^x)$. Therefore $\bar{\mu_k}$ must vanish.

\end{proof}
\begin{remark}
A corollary of Theorem \ref{smoothness} is the moduli space of
$\tau$ semi-stable sheaves on $Y$ is smooth as an Artin stack.
Because the images of $\bar{\mu}_k$ are nothing but obstructions to
smoothness of moduli space.
\end{remark}

We are not going to define Joyce's generalized DT invariants and
state the general form of integrality conjecture since it requires
too much work. The interested readers can refer to \cite{JS} for the
full story.

\begin{corollary}
The generalized Donaldson-Thomas invariants $\hat{DT}(\tau)$ for
$\tau$ semi-stable sheaves are integers on local surfaces.
\end{corollary}
\begin{proof}
The integrality has been proved for the DT invariants of a quiver
without relations. The proof can be found in Theorem 7.28 of
\cite{JS} or \cite{Reineke}. Theorem \ref{smoothness} essentially
says that at every semi-stable point the moduli space
$\mathfrak{M}^\tau(Y,\gamma)$ is locally isomorphic to a moduli
space of representations of a quiver without relations. Therefore,
the integrality follows from \cite{JS} and \cite{Reineke}.
\end{proof}

Pass to quivers is a powerful technique in DT theory. However, in
general this trick works well only when we can match the geometric
stability with appropriate stability on quivers.

In fact, Theorem \ref{moduli of sheaves} can be enhanced to the
following one.

\begin{theorem}
The functor $\RHom_Y(\pi^*(T\otimes L^{-N}),-)$ from
$\mathfrak{M}_\gamma$ to $[MC(L_{\omega,\bd})/G_\bd]$ maps
$\mathfrak{M}^\tau(Y,\gamma)$ to a GIT quotient
$MC(L_{\omega,\bd}){//}_\theta G_\bd$ for suitable choice of
character $\theta$.
\end{theorem}
We won't give a proof to this theorem since it is not used anywhere
in the paper. The idea of the proof is already contained in the work
of King and ´Alvarez-C´onsul \cite{King}.

There are still quite a few open problems on this direction. For
instance, how to describe the geometric stability of objects on $Y$
with noncompact support in term of the quiver stability. The most
important examples are ideal sheaves of curves and stable pairs on
$Y$. In these examples, although the curves are contained in $X$ but
the sheaves, or complexes are not. A natural idea is to define
certain framing conditions on the quiver (We will discuss this in
the later section). However, it is not clear at all why the GIT
stability on framed quiver representations will coincide with the
one on ideal sheaves or stable pairs.

\section{A dimension reduction formula of Virtual motives}\label{sec_motive reduction}
In this section, we give the second application of CS functions. We
will prove a decomposition theorem of virtual motives of $f_\bd$,
which partially generalizes Section 3 of \cite{BBS}. If we could
identify the geometric stability with appropriate quiver stability
condition, then we would obtain a decomposition theorem of virtual
motives of Hilbert schemes, which generalizes the most interesting
part of $\cite{BBS}$. However, so far we have no idea how to deal
with geometric stability.

Let $\LL$ be the motive of affine line. Given a scheme $X$, we will
denote its motive by $[X]$.

Consider a smooth scheme $M$ with an action of algebraic group $G$
together with a $G$-invariant regular function $f: M\to \CC$. In
\cite{DL}, Denef and Loeser defines the motivic vanishing cycle
$[\phi_f]$ in suitable augmented Grothendieck ring of varieites
(called ring of motivic weights). Since our result is not going to
involve the precise definition of this ring, we refer to Section 1
of \cite{BBS} for the precise definitions of ring of motivic
weights.

\begin{definition}\label{virdeg}\cite{BBS}
In appropriate ring of motivic weights, we define the virtual motive
of degeneracy locus by
$$-\LL^{-\frac{dim\ M-dim\
G}{2}}\cdot \frac{[\phi_f]}{[G]}$$ denoted by $[crit^G(f)]_{vir}$.
\end{definition}

We will try to get some property of the virtual motive of the CS
function $f_\bd$. The following Lemma guarantees the main technical
result Proposition 1.11 of \cite{BBS} applies.

\begin{lemma}\label{C*}
Let $f_\bd: L_{\omega,\bd}^1\to \CC$ be the CS function constructed
in Section \ref{sec_CS}. There is a $\CC^*$ action on
$L_{\omega,\bd}^1$ that satisfies:
\begin{enumerate}
\item[(1)] For $\lambda\in \CC^*$, $f_\bd(\lambda \cdot z)=\lambda f_\bd(z)$.
\item[(2)] The limit  $\lim_{\lambda\to 0} \lambda \cdot z$ exists in $L_{\omega,\bd}^1$.
\end{enumerate}
\end{lemma}
\begin{proof}
Let us choose coordinate $z=(y_1,\ldots,y_j,\ldots
w_1^*,\ldots,w_i^*,\ldots)$ on $L_{\omega,\bd}^1$ with respect to
the decomposition $L_{\omega,\bd}^1=L_\bd^1\oplus (L_\bd^2)^*$. As
we remarked in \ref{inner product}, $f_\bd$ equals
$\sum_{i=1}^{dim(L^2_\bd)} a_i(\ldots,y_j,\ldots) w^*_i$ where $a_i$
are polynomials in $y_j$. We define the $\CC^*$ action by scaling
$w^*_i$. The limits of the orbits of this one parameter subgroup is
$L^1_\bd$.
\end{proof}

\begin{theorem}\label{dimension reduction}
Take $X$, $Y$ and $L_\bd$, $L_{\omega,\bd}$ as before. We have the
following dimension reduction formula:
\[
[\phi_{f_\bd}]=-[(L^2_\bd)^*]\cdot[MC(L_\bd)].
\]
\end{theorem}
\begin{proof}
The existence of the $\CC^*$ action in Lemma \ref{C*} implies that
the Milnor fibration given by $f_\bd$ is Zariski trivial outside the
central fiber. Hence
\[
[f_\bd^{-1}(1)]=\frac{[L_{\omega,\bd}^1]-[f_\bd^{-1}(0)]}{(\LL-1)}.
\]

Furthermore, Lemma \ref{C*} together with Proposition 1.11 of
\cite{BBS} implies that
\[
[\phi_{f_\bd}]=[f_\bd^{-1}(1)]-[f_\bd^{-1}(0)]. \]

Recall that
\[
f_\bd=\sum_{i=1}^r a_i(y_1,\ldots,y_j,\ldots)\cdot w_i^*
\]
where $r$ equals the dimension of $L^2_\bd$. We can stratify
$L^1_{\omega,\bd}$ by union of $\{a_i=0|i=1,\ldots,r\}$ and its
complement. The first subscheme is nothing but $MC(L_\bd)$. Using
this stratification, we obtain
\[
[f_\bd^{-1}(0)]=[(L_\bd^2)^*][MC(L_\bd)]+([L^1_\bd]-[MC(L_\bd)])[(L_\bd^2)^*]\LL^{-1}=
(1-\LL^{-1})[(L_\bd^2)^*][MC(L_\bd)]+\LL^{-1}[L^1_{\omega,\bd}].
\]

Then we obtain the formula for $[\phi_{f_\bd}]$
\begin{equation}
\begin{split}
[\phi_{f_\bd}]&=[f_\bd^{-1}(1)]-[f_\bd^{-1}(0)]=-[f^{-1}(0)]\frac{\LL}{\LL-1}+\frac{[L^1_{\omega,\bd}]}{\LL-1}\\
&=-(\frac{\LL-1}{\LL}[(L_\bd^2)^*][MC(L_\bd)]+\frac{L^1_{\omega,\bd}}{\LL})\frac{\LL}{\LL-1}+\frac{[L^1_{\omega,\bd}]}{\LL-1}\\
&=-[(L_\bd^2)^*][MC(L_\bd)].
\end{split}
\end{equation}
\end{proof}

\section{Virtual motives of moduli of representations}\label{sec_generating series}
In this section, we will compute two types of motivic
Donaldson-Thomas invariants: virtual motive of moduli space of
framed representations, which is a noncommutative analogue of
Hilbert schemes and virtual motive of moduli space of semi-stable
representations, which is the noncommutative analogue of moduli
space of semi-stable torsion sheaves.

Without stability, the answer is quite simple as we have seen in
Section \ref{sec_motive reduction}. With stability conditions, the
idea is to write the virtual motive of the moduli stack of all
objects as sum of virtual motives of objects with fixed HN-types.
For each HN-strata, we further decompose the virtual motive into
contributions from semi-stable objects with smaller dimension
vectors. Then try to reverse the formula.

We fix the following notations for motives.
\[
[d]_\LL!:=(\LL^d-1)(\LL^{d-1}-1)\ldots(\LL-1)
\]
\[
\left[\begin{array}{cc}d\\
d'\end{array}\right]:=\frac{[d]_\LL!}{[d-d']_\LL![d']_\LL!}
\]
\[
[\bd]_\LL !:=\prod_{i=0}^n [d_i]_\LL !
\]
\[\left[\begin{array}{cc}\bd\\
\bd'\end{array}\right]_\LL:=\prod_{i=0}^n \left[\begin{array}{cc}d_i\\
d'_i\end{array}\right].
\]
Let $GL_\bd=\prod_{i=0}^n GL_{d_i}$ and $Gr_{\bd',\bd}=\prod_{i=0}^n
Gr(d'_i,d_i)$.

It is easy to show that
\[\displaystyle{[GL_\bd]=\LL^{\sum_{i=0}^n \left(\begin{array}{cc}d_i\\
2\end{array}\right)} [\bd]_\LL !}\]
\[
[Gr_{\bd',\bd}]=\left[\begin{array}{cc}\bd\\
\bd'\end{array}\right]_\LL.
\]
\subsection{Stability conditions}\label{stability}
We follow the notations in section 5.1 of \cite{KSHall}.

Given a quiver $\cQ$, denote the abelian category of finite
dimensional representations of $\cQ$ by $\cA$. We denote its
Grothendieck group by $\Gamma$. A central charge $Z$ (a.k.a.
stability function) is an additive map $Z : \Gamma\simeq\ZZ^{n+1}
\to \CC$ such that the image of any of the standard base vectors
lies in the upper-half plane $\mathbb{H}_+ := \{z\in\CC| Im(z) >
0\}$. Central charge $Z$ is called generic if there are no two
$\QQ$-independent elements of $\ZZ^{n+1}_{\geq 0}$ which are mapped
by $Z$ to the same straight line.

Let us fix a central charge $Z$. Then for any non-zero object
$E\in\cA_\omega$, one defines $Arg(E) := Arg(Z(cl(E)) \in (0, \pi)$
, where $\bd=cl(E)\in\ZZ^{n+1}_{\geq 0}$ is the dimension vector of
the object $E$. We will also use the shorthand notation $Z(E) :=
Z(cl(E))$.

\begin{definition}\label{Z-stable}
A non-zero object $E$ is called semistable (for the central charge
$Z$) if there is no non-zero subobject $F\subset E$ such that
$Arg(F)> Arg(E)$.
\end{definition}

It will be convenient to introduce a total order $\prec$ on the
upper half plane by
\[
z_1\prec z_2\ \iff\ Arg(z_1)>Arg(z_2)\ or\ Arg(z_1)=Arg(z_2)\ and\
|z_1|>|z_2|.
\]
For a nonzero dimension vector $\bd\in \ZZ^{n+1}_{\geq0}$, let us
denote by $\P(\bd)$ the set of collections of dimension vectors
\[\bd_\bullet=(\bd_1,\ldots,\bd_k), k\geq 1,\ such\ that
\sum_{i=1}^k  \bd_i=\bd
\]
where $\bd_i$ are nonzero vector in $\ZZ^{n+1}_{\geq 0}$ satisfying
\[
Arg(Z(\bd_1))>Arg(Z(\bd_2))>\ldots>Arg(Z(\bd_k)).
\]
For any fixed $\bd$, $\P(\bd)$ is a finite set.

We introduce a partial order on $\P(\bd)$ by the formula
\[
(\bd'_1,\ldots,\bd'_{k'})<(\bd_1,\ldots,\bd_k)\ if\ \exists i, 1\leq
i\leq min(k,k')\ such\ that
\]
\[
\bd'_1=\bd_1,\ldots,\bd'_{i-1}=\bd_{i-1}\ and\ Z(\bd'_i)\prec
Z(\bd_i).
\]
This is a total order if the central charge is chosen to be generic.
We will always place this assumption in this paper.

Let $L_{\bd,\bd_\bullet}$ be the constructible subset of
$L_{\omega,\bd}$ consisting of objects $E$ which admit an filtration
$0=E_0\subset E_1\subset\ldots\subset E_k=E$ such that
\[
cl(E_i/E_{i-1})=\bd_i, i=1,\ldots, k.
\] and $\bd_\bullet\in \P(\bd)$. And let $L_{\bd,\leq \bd_\bullet}$ be
$\bigcup_{\bd'_\bullet\leq \bd_\bullet} L_{\bd,\bd'_\bullet}$.

When $\bd'_\bullet<\bd_\bullet$, there is an inclusion from
$L_{\bd,\bd'_\bullet}$ to $L_{\bd,\bd_\bullet}$ given by forgetting
the finer part of filtration. If we list the elements of $\P(\bd)$
as
\[
\bd^{(N)}_\bullet<\ldots<\bd^{(2)}_\bullet<\bd^{(1)}_\bullet<\bd^{(0)}_\bullet=\bd_\bullet,
\]
then the sequence $L_{\bd,\leq \bd^{(i)}_\bullet}, i=0,\ldots, N$
for all $\bd^i_\bullet\in \P(\bd)$ is a decreasing sequence of
closed subschemes in $L_\bd$. Their compliments form a chain of open
subschemes
\[
\emptyset=U_0\subset U_1\subset\ldots\subset U_N=L_\bd.
\]
The consecutive differences $U_i-U_{i-1}$ are locally closed smooth
subschemes, denoted by $L^{HN}_{\bd,\bd_\bullet}$, i.e. the locus of
objects with dimension vector $\bd$ and fixed Harder-Narisimanhan
type $\bd_\bullet$. In particular, $U_1$ is the subscheme of
semi-stable objects.

\subsection{Framed representations}\label{NCHilb}
\begin{definition}
Take the quiver $\cQ_\omega$ defined in the previous section. Given
a dimension vector $\bd$, let $V_0,\ldots, V_n$ be the sequence of
vector spaces of dimension $d_0,\ldots,d_n$ over the nodes. A framed
representation $V$ of $\cQ_\omega$ with dimension vector $\bd$ is a
representation of $\cQ_\omega$ together with a vector
$v=(v_0,\ldots,v_n)$ such that $v_i\in V_i$. A framed representation
$V$ is called cyclic if $v_0,\ldots,v_n$ generate $V$.
\end{definition}

Denote the submodule generated by $v$ by $M_v$. Let
\[
Y_{\bd}=\{(A,v)\in L^1_{\omega,\bd}\times V_0\times\ldots\times V_n:
f_\bd=0\}
\]
and let
\[
Z_{\bd}=\{(A,v)\in L^1_{\omega,\bd}\times V_0\times\ldots\times V_n:
f_\bd=1\}.
\]
 Then
$Y_{\bd}=\bigsqcup_{\bd'<\bd} Y_{\bd}^{\bd'}$ and
$Z_{\bd}=\bigsqcup_{\bd'<\bd} Z_{\bd}^{\bd'}$, where
$$
Y_{\bd}^{\bd'}=\{(A,v)\in L^1_{\omega,\bd}\times
V_0\times\ldots\times V_n: f_\bd=0, cl(M_v)=\bd'\},$$
$$Z_{\bd}^{\bd'}=\{(A,v)\in L^1_{\omega,\bd}\times V_0\times\ldots\times V_n: f_\bd=1, cl(M_v)=\bd'\}.
$$
Writing $w_{\bd}=[Y_{\bd}]-[Z_{\bd}]$ and
$w_{\bd}^{\bd'}=[Y_{\bd}^{\bd'}]-[Z_{\bd}^{\bd'}]$.

Let $|\bd|=\sum_{i=0}^n d_i$. By Theorem \ref{dimension reduction},
we have
\begin{equation}\label{w_d-C_d}
w_\bd=\LL^{|\bd|}[(L_\bd^2)^*][MC(L_\bd)].
\end{equation}
There is a projection from $Y_\bd^{\bd'}$ to the Grassmannian
$Gr_{\bd',\bd}$, whose fiber is the set
\[\{(\left(\begin{array}{cc}A^0 &A'\\0& A^1
\end{array}\right), v
)|f_\bd=0\}\] where $A^0$ are matrices of size $\bd'\times\bd'$
(depending on the source and target vertices), $A^1$ are matrices of
size $(\bd-\bd')\times (\bd-\bd')$ and $A'$ are matrices of size
$(\bd-\bd')\times\bd'$. There is an embedding from
$L^1_{\omega,\bd'}\times L^1_{\omega,\bd-\bd'}$ into
$L^1_{\omega,\bd}$ by mapping to block diagonal matrices.

The CS function $f_\bd$ satisfies
\[
f_\bd(\left(\begin{array}{cc}A^0 &A'\\0& A^1
\end{array}\right), v)=f_{\bd'}(A^0,v)+f_{\bd-\bd'}(A^1,v).
\]
Denote the subgroup of $GL_\bd$ that preserves these Borel matrices
by $B_{\bd,\bd'}$ and the Euler form of $\cQ_\omega$ by $\chi$.
\[
[Y_\bd^{\bd'}]= \frac{[B_{\bd,\bd'}]}{[GL_{\bd'}][GL_{\bd-\bd'}]}\cdot\LL^{-\chi(\bd',\bd-\bd')} \left[\begin{array}{cc}\bd\\
\bd'\end{array}\right]_\LL([Y_{\bd'}^{\bd'}]\cdot[Y_{\bd-\bd'}]+(\LL-1)\cdot[Z_{\bd'}^{\bd'}]\cdot[Z_{\bd-\bd'}])
\cdot \LL^{-|\bd-\bd'|}.
\]
A similar analysis yields
\[
[Z_\bd^{\bd'}]=\frac{[B_{\bd,\bd'}]}{[GL_{\bd'}][GL_{\bd-\bd'}]}\cdot\LL^{-\chi(\bd',\bd-\bd')}\left[\begin{array}{cc}\bd\\
\bd'\end{array}\right]_\LL ([Y_{\bd'}^{\bd'}]\cdot[Z_{\bd-\bd'}]+
(\LL-2)\cdot[Z_{\bd'}^{\bd'}]\cdot[Z_{\bd-\bd'}]+
[Z_{\bd'}^{\bd'}]\cdot[Y_{\bd-\bd'}])\cdot\LL^{-|\bd-\bd'|}.
\]
The above formulas, combining with equation \ref{w_d-C_d}, yields
\begin{equation}\label{w_d_d'}
\begin{split}
w_\bd^{\bd'} &= \frac{[B_{\bd,\bd'}]}{[GL_{\bd'}][GL_{\bd-\bd'}]}\LL^{-\chi(\bd',\bd-\bd')}\LL^{-|\bd-\bd'|} \left[\begin{array}{cc}\bd\\
\bd'\end{array}\right]_\LL (w_{\bd'}^{\bd'}\cdot w_{\bd-\bd'})\\
&=\frac{[B_{\bd,\bd'}][(L_{\bd-\bd'}^2)^*]}{[GL_{\bd'}][GL_{\bd-\bd'}]}\LL^{-\chi(\bd',\bd-\bd')}\left[\begin{array}{cc}\bd\\
\bd'\end{array}\right]_\LL [MC(L_{\bd-\bd'})]\cdot w^{\bd'}_{\bd'}.
\end{split}
\end{equation}
Following from the fact that $Y_{\bd}=\bigsqcup_{\bd'<\bd}
Y_{\bd}^{\bd'}$ and $Z_{\bd}=\bigsqcup_{\bd'<\bd} Z_{\bd}^{\bd'}$,
we get
\[
w^\bd_\bd=w_\bd-\sum_{\bd'<\bd} w^{\bd'}_\bd.
\]
Let $\tilde{c}_\bd=\frac{[MC(L_\bd)]}{[GL_\bd]}$. Applying
\ref{w_d-C_d} and \ref{w_d_d'}, we obtain a recursion formula
\begin{equation}\label{recursion}
\begin{split}
w^\bd_\bd&=\LL^{|\bd|}[(L_{\bd}^2)^*][MC(L_\bd)]-\sum_{\bd'<\bd}
\frac{[B_{\bd,\bd'}][(L_{\bd-\bd'}^2)^*]}{[GL_{\bd'}][GL_{\bd-\bd'}]}
\cdot\LL^{-\chi(\bd',\bd-\bd')}\left[\begin{array}{cc}\bd\\
\bd'\end{array}\right]_\LL [MC(L_{\bd-\bd'})]\cdot w^{\bd'}_{\bd'}\\
-\frac{[\phi_{f^{ss}_\bd}]}{[GL_\bd]}&=\LL^{|\bd|}[(L_{\bd}^2)^*]\tilde{c}_\bd+\sum_{\bd'<\bd}
[(L_{\bd-\bd'}^2)^*] \cdot\LL^{-\chi(\bd',\bd-\bd')}
\tilde{c}_{\bd-\bd'}\cdot \frac{[\phi_{f^{ss}_{\bd'}}]}{[GL_{\bd'}]}
\end{split}
\end{equation}
Here $f^{ss}_\bd$ is the restriction of $f_\bd$ to the semi-stable
loci.

Define the virtual motive of the noncommutative Hilbert scheme
$Hilb^\bd$ by
\begin{equation}\label{virtual Hilb}
[Hilb^\bd]_{vir}:=-\LL^{\frac{\chi(\bd,\bd)-|\bd|}{2}}
\frac{[\phi_{f^{ss}_\bd}]}{[GL_\bd]}.
\end{equation}
Replace $\phi_{f^{ss}_\bd}$ by $[Hilb^\bd]_{vir}$ subject to the
above formula, we obtain
\begin{equation}
\begin{split}
\LL^{|\bd|}[(L_{\bd}^2)^*]\tilde{c}_\bd&=\sum_{\bd'\leq\bd}
\LL^{-\frac{2\chi(\bd',\bd-\bd')+\chi(\bd',\bd')}{2}}\LL^{\frac{|\bd'|}{2}}
[(L_{\bd-\bd'}^2)^*] \tilde{c}_{\bd-\bd'}\cdot [Hilb^{\bd'}]_{vir}\\
\LL^{\frac{|\bd|}{2}}[(L_{\bd}^2)^*]\tilde{c}_\bd&=\sum_{\bd'\leq\bd}
\LL^{-\frac{\chi(\bd,\bd)-\chi(\bd-\bd',\bd-\bd')}{2}}\LL^{-\frac{|\bd-\bd'|}{2}}
[(L_{\bd-\bd'}^2)^*] \tilde{c}_{\bd-\bd'}\cdot [Hilb^{\bd'}]_{vir}\\
\LL^{\frac{\chi(\bd,\bd)}{2}}\LL^{\frac{|\bd|}{2}}[(L_{\bd}^2)^*]\tilde{c}_\bd&=\sum_{\bd'\leq\bd}
\LL^{\frac{\chi(\bd-\bd',\bd-\bd')}{2}}\LL^{-\frac{|\bd-\bd'|}{2}}
[(L_{\bd-\bd'}^2)^*] \tilde{c}_{\bd-\bd'}\cdot [Hilb^{\bd'}]_{vir}
\end{split}
\end{equation}
Define the generating series for $\widetilde{c}_\bd$ by
\[
C(t)=\sum_{\bd\in\ZZ_{\geq0}^{n+1}}
\LL^{\frac{\chi(\bd,\bd)}{2}}[(L_{\bd}^2)^*]\widetilde{c}_\bd\cdot
t^\bd
\] and the generating series of noncommutative Hilbert schemes by
\[
Z(t)=\sum_{\bd\in\ZZ_{\geq0}^{n+1}}[Hilb^\bd]_{vir}\cdot t^\bd.
\]
Then the generating series of Hilbert schemes can be written as
\begin{equation}\label{formula_Hilb}
Z(t)=\frac{C(\LL^{\frac{1}{2}}t)}{C(\LL^{-\frac{1}{2}}t)}.
\end{equation}

Finally, notice that $\LL^{\frac{\chi(\bd,\bd)}{2}}[(L_{\bd}^2)^*]$
is nothing but $\LL^{\frac{\chi_\cQ(\bd,\bd)}{2}}$ for the Euler
form of the quiver $\cQ$. So $C(t)$ is the generating series of
moduli space of representations of $\cQ$ (without stability).
\subsection{(Unframed) Semi-stable representations}\label{SStorison}
The goal of this section is to compute a recursion formula of the
virtual motive of semi-stable representations of $\cQ_\omega$
without framing condition.

Recall that an element $\bd_\bullet\in\P(\bd)$ is a collection of
dimension vectors $\bd_1,\ldots,\bd_k$ such that
$\bd_1+\ldots+\bd_k=\bd$ and
\[ Arg(Z(\bd_1))>Arg(Z(\bd_2))>\ldots>Arg(Z(\bd_k)).
\]
Let $F:=\{V^j_i\}$ be a flag of vector spaces over node
$i=0,\ldots,n$. Let $Fl_{\bd_\bullet}$ be the flag variety whose
points are flags $0=V^0_i\subset V^1_i\subset\ldots V^{k-1}_i\subset
V^k_i=V_i$ for $i=0,\ldots,n$, such that the dimension vector of
$(V^j_i)_i$ is equal to $\sum_{i=1}^j\bd_i$. Consider the incidence
subvariety $I_{\bd,\bd_\bullet}=\{(A,F)|A\ preserves\ F\}\subset
L_{\bd}\times Fl_{\bd_\bullet}$. Its projection to $L_{\omega,\bd}$
is nothing but $L_{\bd,\bd_\bullet}$.

The Harder-Narisimanhan locus $L^{HN}_{\bd,\bd_\bullet}$ is defined
to be
\[L_{\bd,\bd_\bullet}-\cup_{\bd'_\bullet<\bd_\bullet}L_{\bd,\bd'_\bullet}.\]
Suppose the central charge is generic, all the possible
$\bd'_\bullet$ that are smaller than $\bd_\bullet$ can be put in a
linear sequence of inequalities:
\[
\bd^{(N)}_\bullet<\ldots<\bd^{(2)}_\bullet<\bd^{(1)}_\bullet<\bd^{(0)}_\bullet=\bd_\bullet.
\]
Then we obtain $L_{\bd,\bd_\bullet}$ as a disjoint union of HN-loci
\begin{equation}\label{HN-decomposition}
L_{\bd,\bd_\bullet}=\bigsqcup_{i=0}^N
L^{HN}_{\bd,\bd^{(i)}_\bullet}. \end{equation}

In particular, we could take $\bd_\bullet=\bd$, then
$\bd^{(0)}_\bullet=\bd$. We rewrite the maximal strata
$L^{HN}_{\bd,\bd^{(0)}_\bullet}$ as $L^{ss}_{\omega, \bd}$.

The incidence subvariety $I_{\bd,\bd_\bullet}$ has an open
subvariety $I^{ss}_{\bd,\bd_\bullet}$ consisting of objects with
semi-stable subquotients. The fibers of the projection
$I^{ss}_{\bd,\bd_\bullet}\to Fl_{\bd,\bd_\bullet}$ can be identified
with $U^{ss}_{\bd,\bd_\bullet}$ which consists of block upper
triangular matrices of types given by $\bd_\bullet$. The product
$\prod_{j=1}^{k_i} L^{ss}_{\omega,\bd_j}$ embeds into
$U^{ss}_{\bd,\bd_\bullet}$ as block diagonal matrices. The pair of
algebraic groups $(B_{\bd,\bd_\bullet},GL_{\bd_\bullet})$ acts on
$(U^{ss}_{\bd,\bd_\bullet},\prod_{j=1}^{k_i} L^{ss}_{\omega,\bd_j})$
in a compatible way, where $B_{\bd,\bd_\bullet}$ is the Borel group
preserving $U^{ss}_{\bd,\bd_\bullet}$ and $GL_{\bd_\bullet}$ is the
product $\prod_{j=1}^{k_i} GL_{\bd_j}$. The Harder-Narisimanhan
locus $L^{HN}_{\bd,\bd_\bullet}$ equals to the image of projection
from $I^{ss}_{\bd,\bd_\bullet}$ to $L_{\omega,\bd}$, which can be
identified with $GL_\bd\times_{B_{\bd,\bd_\bullet}}
U^{ss}_{\bd,\bd_\bullet}$. Then
\[
[U^{ss}_{\bd,\bd_\bullet}]=\frac{[B_{\bd,\bd_\bullet}]}{[GL_{\bd_\bullet}]}\cdot
\LL^{-\sum_{j<j'}\chi(\bd_j,\bd_{j'})} \prod_{j=1}^{k_i}
[L^{ss}_{\omega,\bd_j}]
\]
\[
[L^{HN}_{\bd,\bd_\bullet}]=\LL^{-\sum_{j<j'}\chi(\bd_j,\bd_{j'})}\frac{[GL_\bd][U^{ss}_{\bd,\bd_\bullet}]}{[B_{\bd,\bd_\bullet}]}
=\LL^{-\sum_{j<j'}\chi(\bd_j,\bd_{j'})}\frac{[GL_\bd]\prod_{j=1}^{k_i}
[L^{ss}_{\omega,\bd_j}]}{\prod_{j=1}^{k_i} [GL_{\bd_j}]}
\]

The above situation can be thought as representations of
$\cQ_\omega$ forgetting the relations. Now we will pass to quiver
with relations. Geometrically, we intersect them with $crit(f_\bd)$.

\paragraph{Genericity assumption}
Assume $Z$ is chosen in such a way that $L^{HN}_{\bd,\bd_\bullet}$
is transversal to critical set of the CS function $f_\bd$.\\

Denote the restriction of $f_\bd$ to $L^{HN}_{\bd,\bd_\bullet}$ by
$f_{\bd_\bullet}$ and the restriction to $L^{ss}_{\omega,\bd}$ by
$f^{ss}_\bd$.
\[[\phi_{f_{\bd_\bullet}}]=[\phi_{f|_{U^{ss}_{\bd,\bd_\bullet}}}]\cdot\LL^{-\sum_{j<j'}\chi(\bd_j,\bd_{j'})}
\frac{[GL_\bd]}{[B_{\bd,\bd_\bullet}]}.\]

Under the genericity assumption, every HN-strata
$L^{HN}_{\bd,\bd_\bullet}$ is transversal to $crit(f_\bd)$. The
restriction $f_{\bd_\bullet}$ also has the $C^*$-action mentioned in
Lemma \ref{C*}. Therefore Proposition 1.11 of \cite{BBS} applies.
This leads to the following equation:
\[
[\phi_{f_\bd}]=\sum_{i=0}^N
[\phi_{f|_{L^{HN}_{\omega,\bd,\bd^{(i)}_\bullet}}}].
\]

The restriction $f|_{U^{ss}_{\bd,\bd^{(i)}_\bullet}}$ only depends
on the diagonal parts. By Thom-Sebastiani theorem, we obtain
\[
[\phi_{f|_{U^{ss}_{\bd,\bd^{(i)}_\bullet}}}]=[\phi_{f^{ss}_{\bd_1}}]
\star\ldots\star[\phi_{f^{ss}_{\bd_{k_i}}}]\cdot\frac{[B_{\bd,\bd_\bullet}]}{[GL_{\bd_\bullet}]}.
\]
After passing to the quotient stacks, we obtain
\[
\frac{[\phi_{f_{\bd^{(i)}}}]}{[GL_{\bd^{(i)}}]}=\LL^{-\sum_{j<j'}\chi(\bd_j,\bd_{j'})}\frac{[\phi_{f^{ss}_{\bd_1}}]}{[GL_{\bd_1}]}
\star\ldots\star\frac{[\phi_{f^{ss}_{\bd_{k_i}}}]}{[GL_{\bd_{k_i}}]}.
\]

Combing the above computations, we have
\[
[\phi_{f_\bd}]=\sum_{i=0}^N
[\phi_{f|_{L^{HN}_{\omega,\bd,\bd^{(i)}_\bullet}}}]=\sum_{i=0}^N
\LL^{-\sum_{j<j'}\chi(\bd_j,\bd_{j'})}[\phi_{f^{ss}_{\bd_1}}]
\star\ldots\star[\phi_{f^{ss}_{\bd_{k_i}}}]\cdot\frac{[GL_{\bd}]}{[GL_{\bd^{(i)}_\bullet}]}.
\]

\[
-[(L^2_\bd)^*]\cdot[MC(L_\bd)]=\sum_{i=0}^N
\LL^{-\sum_{j<j'}\chi(\bd_j,\bd_{j'})}[\phi_{f^{ss}_{\bd_1}}]
\star\ldots\star[\phi_{f^{ss}_{\bd_{k_i}}}]\cdot\frac{[GL_{\bd}]}{[GL_{\bd^{(i)}_\bullet}]}.
\]

\[
-\sum_\bd [(L^2_\bd)^*]\cdot \tilde{c}_\bd t^\bd=\sum_\bd
\sum_{i=0}^N
\LL^{-\sum_{j<j'}\chi(\bd_j,\bd_{j'})}\prod_{j=1}^{k_i}[\phi_{f^{ss}_{\bd_j}}]
t^\bd .
\]

Now we define the virtual motive of semi-stable representations with
dimension vector $\bd$ by
\[
[\mathfrak{M}^{ss}_\bd)]_{vir}:=-\LL^{\frac{\chi(\bd,\bd)}{2}}\cdot\frac{[\phi_{f^{ss}_\bd}]}{[GL_\bd]}.
\]
It follows that
\begin{equation}
\begin{split}
[(L^2_\bd)^*]\LL^{\frac{\chi(\bd,\bd)}{2}}\widetilde{c}_\bd&=\sum_{i=0}^N\prod_{j=1}^{k_i}[\mathfrak{M}^{ss}_\bd)]_{vir}\\
\sum_{\bd\in\ZZ_{\geq0}^{n+1}}\LL^{\frac{\chi_{\cQ}(\bd,\bd)}{2}}\widetilde{c}_\bd
t^\bd&=\sum_{\bd\in\ZZ_{\geq0}^{n+1}}(\sum_{i=0}^N\prod_{j=1}^{k_i}[\mathfrak{M}^{ss}_\bd)]_{vir})
t^\bd\\
C(t)&=\sum_{\bd\in\ZZ_{\geq0}^{n+1}}\sum_{i=0}^N(\prod_{j=1}^{k_i}[\mathfrak{M}^{ss}_\bd)]_{vir})
t^{\bd^{(i)}}.
\end{split}
\end{equation}
If we write the generating series of semi-stable representations as
$Z^{ss}(t)=\sum_{\bd\in\ZZ_{\geq0}^{n+1}}[\mathfrak{M}^{ss}_\bd)]_{vir}\cdot
t^\bd$, then the above formula can be rewritten as
\begin{equation}
C(t)=\overrightarrow{\prod} Z^{ss}(t).
\end{equation}
The product is the ordered product respect to the phase of the
central charge. The righthand side involves invariants of the quiver
$\cQ_\omega$ with stability and the lefthand side is the generating
series of $\cQ$ without stability. It is not clear how to reverse
this formula.

\end{document}